\documentclass[12pt,a4paper]{article}
\usepackage{amsmath,amsfonts,amssymb,amsthm,amscd}
\newtheorem{thm}{Theorem}
\newtheorem{lem}[thm]{Lemma}
\newtheorem{prop}[thm]{Proposition}
\newtheorem{defn}[thm]{Definition}

\newtheorem{rem}[thm]{Remark}
\newtheorem{notation}[thm]{Notations}
\newtheorem{ex}[thm]{Example}
\newtheorem{assumption}[thm]{Assumption}

\begin{document}

\title{On adding a variable to a Frobenius manifold and generalizations}

\author{Liana David}

\maketitle

{\bf Abstract:} Let $\pi :V\rightarrow M$ be a (real or
holomorphic) vector bundle whose base has an almost Frobenius
structure $(\circ_{M},e_{M},g_{M})$ and typical fiber has the
structure of a Frobenius algebra $(\circ_{V},e_{V},g_{V})$. Using
a connection $D$ on the bundle $V$ and a morphism $\alpha
:V\rightarrow TM$, we construct an almost Frobenius structure
$(\circ , e_{V},g)$ on the manifold $V$ and we study when it is
Frobenius. In particular, we describe all (real) positive definite
Frobenius structures on $V$ obtained in this way, when $M$ is a
semisimple Frobenius manifold with non-vanishing rotation
coefficients. In the holomorphic setting, we add a real structure
$k_{M}$ on $M$ and a real structure $k_{V}$ on the fibers of $\pi$
and we study when an induced real structure on $V$, together with
the almost Frobenius structure $(\circ , e_{V}, g)$, satisfy the
$tt^{*}$-equations. Along the way, we prove various properties of
adding variables to
a Frobenius manifold, in connection with Legendre transformations and $tt^{*}$-geometry.\\



\section{Introduction}

Frobenius manifolds were defined by B. Dubrovin \cite{dubrovin},
as a geometrical interpretation of the so called WDVV
(Witten-Dijkgraaf-Verlinde-Verlinde) equations. They also appear
in many different areas of mathematics - singularity theory
\cite{hertling-book}, quantum cohomology \cite{manin-book} and
integrable systems \cite{hitchin}, providing an unexpected link
between these apparently different fields. Rather than the
original definition of Dubrovin, along the paper we shall use the
following alternative definition of Frobenius manifolds (the
equivalence of the two definitions is a consequence of Theorem
2.15 and Lemma 2.16 of \cite{hertling-book}).

\begin{defn}
1) An almost Frobenius structure on a manifold $M$ in one of the
standard categories (smooth or holomorphic) is given by a
(fiber preserving) commutative, associative multiplication $\circ$ on $TM$, with unit
field $e$, and a metric $g$ on $M$, invariant with respect to
$\circ$, i.e. such that
$$
g(X\circ Y, Z) = g(X,Y\circ Z)\quad\forall X,Y,Z\in {\mathcal X}(M).
$$
2) An almost Frobenius structure $(\circ ,e, g)$ on $M$ is called
Frobenius (and $(M, \circ , e, g)$ is a Frobenius manifold)
if the following conditions hold:\\

i) $(M,\circ , e)$ is an $F$-manifold, that is,
$$
L_{X\circ Y}(\circ ) = X\circ L_{Y}(\circ ) + Y\circ L_{X}(\circ ),
\quad \forall X,Y\in{\mathcal X}(M).
$$
ii) the metric $g$ is admissible on the $F$-manifold $(M,\circ ,e)$ (i.e.
the unit field $e$ is parallel with respect to the
Levi-Civita connection of $g$)
and is flat.
\end{defn}

In flat coordinates $(t^{i})$ for the Frobenius metric $g$ one may
write
$$
g\left( \frac{\partial}{\partial
t^{i}}\circ\frac{\partial}{\partial t^{j}},
\frac{\partial}{\partial t^{k}}\right) =
\frac{\partial^{3}F}{\partial t^{i}\partial t^{j}\partial t^{k}},
$$
for a function $F$, called the potential of the Frobenius manifold
and defined up to adding a quadratic expression in $t^{i}$. The
asssociativity of $\circ$ translates into the WDVV-equations for
$F$: for any fixed $i,j,k,r$,
$$
\sum_{m,s}\frac{\partial^{3}F}{\partial t^{i}\partial
t^{j}\partial t^{s}}g^{sm} \frac{\partial^{3}F}{\partial
t^{m}\partial t^{k}\partial t^{r}} =
\sum_{m,s}\frac{\partial^{3}F}{\partial t^{k}\partial
t^{j}\partial t^{s}} g^{sm}\frac{\partial^{3}F}{\partial
t^{m}\partial t^{i}\partial t^{r}}.
$$

In examples arising from singularity theory, a Frobenius manifold
comes naturally equipped with an Euler field.

\begin{defn} An Euler field on a Frobenius manifold $(M,\circ , e, g)$
is a vector field $E$ such that
$$
L_{E}(\circ )= \circ ,\quad L_{E}(g) = dg
$$
where $d$ is a constant (equal to $2$, when $g(e,e)\neq 0$).
\end{defn}

{\bf Outline of the paper.} It is customary in modern mathematics
to consider various geometrical structures on manifolds and to
search for similar structures on related manifolds (like
submanifolds, quotient manifolds, total spaces of vector bundles,
etc). The starting point of this paper is a result proved in
Chapter VII of \cite{sabbah}, which states that there is a natural
Frobenius structure on the product $M\times \mathbb{K}$ when
$(M,\circ ,e, g)$ is Frobenius $(\mathbb{K} = \mathbb{R}$ when $M$
is a real manifold and $\mathbb{K}=\mathbb{C}$ when $M$ is
complex). It is usually referred as {\it the Frobenius structure
obtained by adding a variable to $M$}, and has unit field
$\frac{\partial}{\partial\tau}$ (where $\tau$ is the coordinate on
$\mathbb{K}$). It is important to note that the Frobenius metric
on $M\times\mathbb{K}^{r}$  is not a product metric, and neither
the multiplication is a product multiplication. If $E$ is an Euler
field on $M$ with $L_{E}(g) =2g$, then $E +R$ (where $R_{(p,\tau
)}= \tau\frac{\partial}{\partial \tau}$, $(p,\tau )\in
M\times\mathbb{K}$ is the radial field) is Euler on
$M\times\mathbb{K}.$ The overall aim of this paper is to develop
generalizations of this construction.

Section \ref{preliminary}, intended mostly for completeness (only
Proposition \ref{iterations} is original) contains a brief account
of the Frobenius structure (with or without an Euler field)
obtained from a vector bundle with additional data (e.g. a Saito
structure) and a primitive section \cite{saito}. Such a data
appears naturally in the theory of isomonodromic deformations.
Then we recall, following \cite{sabbah}, the construction of
adding a variable to a Frobenius manifold. With this preliminary
background, we determine the general form of the Frobenius
structure on $M\times\mathbb{K}^{r}$ ($r\geq 1$), obtained by
Legendre transformations and successively adding variables to a
Frobenius manifold $M$ (see Proposition \ref{iterations}), this
being a good motivation for our treatment from the following
sections.

The general set-up of the paper is a $\mathbb{K}$-vector bundle
$\pi: V\rightarrow M$, with a connection $D$, an associative,
commutative multiplication $\circ_{M}$ on $TM$, with unit field
$e_{M}$, a similar multiplication $\circ_{V}$ on the fibers of
$V$, with unit field $e_{V}\in\Gamma (V)$, and a morphism $\alpha
: V\rightarrow TM.$ From this data we construct a multiplication
$\circ$ on $TV$, with unit field $e_{V}$ (viewed as a
tangent vertical vector field). Later on, we  add to this picture
an invariant metric $g_{M}$ on $(M,\circ_{M},e_{M})$ and an
invariant metric $g_{V}$ on the fibers of $V$, and we construct a
metric $g$ on the manifold $V$. (The case when $V$ is  the trivial
bundle of rank one, $D$ is the trivial connection, $\alpha
(e_{V})=e_{M}$ and $(M,\circ_{M},e_{M},g_{M})$ is Frobenius
corresponds to adding a variable to $M$). We determine various
properties of the structures on $V$ so defined, as follows.

In Section \ref{four1} we
prove that $(V,\circ ,e_{V})$ is an $F$-manifold
if and only if
$(M,\circ_{M},e_{M})$ is an $F$-manifold, the connection $D$ is flat,
the multiplication $\circ_{V}$ is $D$-parallel and two other
natural conditions are satisfied, namely  conditions (\ref{alpha})
and (\ref{croset}) (see Proposition \ref{multiplication}).
We remark that the commutativity condition (\ref{croset})
is implied by (\ref{alpha}), when the $F$-manifold $(M,\circ_{M} , e_{M})$ is
semisimple (see Remark \ref{rem-ss}).

In Section \ref{admissible} we assume that $(V, \circ , e_{V})$ is
an $F$-manifold and we prove that $g$ is an admissible metric on
$(V,\circ ,e_{V})$ if and only if $D(g_{V})=0$ and the coidentity
$\epsilon_{M}:= g_{M}(e_{M}, \cdot )$ is closed (see Proposition
\ref{admissible-cond}).

In Section \ref{frobenius}  we assume that $(V,\circ ,e_{V},g)$ is
an $F$-manifold with admissible metric and that the base $(M,
\circ_{M},e_{M},g_{M})$ is a Frobenius manifold and we show that
the flatness of $g$ (i.e. the only remaining condition for $(V,
\circ , e_{V}, g)$ to be Frobenius) can be expressed in terms of a
system of conditions for the morphism $\alpha$, see Proposition
\ref{cond-flatness}. Our main result in this section is Theorem
\ref{corectat}, which provides a complete description, in the real case, of all Frobenius structures
$(\circ , e_{V}, g)$ on $V$, with positive definite metric $g$,
when $(M,\circ_{M},e_{M}, g_{M})$ is a semisimple Frobenius manifold with non-vanishing rotation coefficients.
It turns out that
\begin{equation}\label{forma-alpha}
\alpha = \lambda\otimes e_{M}
\end{equation}
with $\lambda\in \Gamma (V^{*})$ a $D$-parallel section
satisfying
$$
\lambda (e_{V})=1,\quad \lambda (v_{1}\circ_{V}v_{2}) = \lambda (v_{1})\lambda (v_{2}),\quad\forall v_{1},v_{2}\in V.
$$
It would be interesting to understand if the morphism $\alpha$
must be of this form also in other signatures, or in the complex
case. At the end of this section we return to the general picture
(no semisimplicity assumptions) and we consider a class of
Frobenius structures $(\circ ,e_{V}, g)$ on the total space of the
trivial bundle $\pi :V= M\times\mathbb{K}^{r}\rightarrow M$
obtained from the data $(D,\circ_{M}, e_{M},\circ_{V}, e_{V},
\alpha , g_{M}, g_{V})$ as before, where the base $(M,
\circ_{M},e_{M},g_{M})$ is a Frobenius manifold, $D$ is the
trivial standard connection and $\alpha$ is of the form
(\ref{forma-alpha}) (see Theorem \ref{trivial}). We study in
detail this class of Frobenius manifolds, in connection with Euler
fields and Saito bundles (see Proposition \ref{trivial-prop}) and
we compute their potential (see Remark \ref{potential}).

In Section \ref{sect-tt} we return to the general setting of an
holomorphic vector bundle $\pi : V\rightarrow M$ with data
$(D,\circ_{M},e_{M}, \circ_{V},e_{V}, \alpha , g_{M}, g_{V})$. Let
$\circ$ and $g$ be the associated multiplication and metric on the
manifold $V$, defined by this data. (We do not assume that $(\circ
,e_{V}, g)$ is a Frobenius structure). Instead, we assume that $M$
has a real structure $k_{M}$ and the fibers of $\pi$ have a real
structure $k_{V}$. Using $k_{M}$ and $k_{V}$ we define a real
structure $k$ on $V$ (which coincides with $k_{M}$ on the
$D$-horizontal subbundle and with $k_{V}$ on the vertical
subbundle). Our main result is Theorem \ref{main-thm-tt}, which
states the necessary and sufficient conditions for the
$tt^{*}$-equations to hold on $(V, \phi , h)$ where $\phi_{X}Y= -
X\circ Y$ and $h(X,Y) = g(X,kY).$ Then we discuss an example where
all these conditions hold (see Example \ref{detailed-ex}).
Finally, we return to the Frobenius structure from Theorem
\ref{trivial} (with $\mathbb{K}=\mathbb{C})$ and we show that
imposing the $tt^{*}$-equations in the framework of this section
gives high restrictions on the geometry of the base $(M,
\circ_{M},e_{M},g_{M})$ (see Proposition \ref{second} and the
comments before this proposition).\\

{\bf Acknowledgements:} This work was supported by a grant of the Romanian National Authority
for Scientific Research, CNCS-UEFISCDI, project no.
PN-II-ID-PCE-2011-3-0362.

\section{Legendre transformations and adding variables to
Frobenius manifolds}\label{preliminary}

Along the paper we  use the following conventions.  A
$\mathbb{K}$-manifold is a real manifold when $\mathbb{K}
=\mathbb{R}$ or a complex manifold when $\mathbb{K}= \mathbb{C}.$
We  denote by ${\mathcal X}(M)$ the sheaf of $\mathbb{K}$-vector
fields on a $\mathbb{K}$-manifold $M$ (i.e. smooth vector fields,
if $M$ is a real manifold, or holomorphic vector fields, when $M$
is a complex manifold). By a $\mathbb{K}$-vector bundle we mean a
real vector bundle ($\mathbb{K}=\mathbb{R}$) or an holomorphic
vector bundle ($\mathbb{K} =\mathbb{C}$). Unless otherwise stated,
all objects we consider on $\mathbb{C}$-vector bundles (e.g.
sections, metrics, connections, Higgs fields, endomorphisms, etc)
are holomorphic and  for a complex manifold $M$, $TM$ will denote
the holomorphic tangent bundle and $\Omega^{1}(M, V)$ the bundle
of holomorphic $1$-forms with values in a $\mathbb{C}$-vector
bundle $V\rightarrow M.$ In our conventions, the curvature $R^{D}$
of a connection $D$ is given by $R^{D}(X,Y) =D_{X}D_{Y} -
D_{Y}D_{X} - D_{[X,Y]}$.

\subsection{Frobenius structures from infinitesimal period
mappings}

Let $\pi :V\rightarrow M$ be a $\mathbb{K}$-vector bundle
with $\mathrm{rank}(V) = \mathrm{dim}(M)$, endowed
with a connection
$\nabla$, a metric $g$ and a vector valued $1$-form $\phi \in \Omega^{1}(M,
\mathrm{End}V)$, satisfying the conditions:
\begin{equation}\label{saito0}
R^{\nabla}=0,\quad d^{\nabla}\phi =0,\quad \nabla g=0,\quad
\phi\wedge \phi =0,\quad \phi^{*}=\phi ,
\end{equation}
where
$$
(\phi\wedge \phi )_{X,Y} := \phi_{X}\phi_{Y}-\phi_{Y}\phi_{X},\quad X,Y\in TM
$$
and, for any $X\in TM$, $\phi^{*}_{X}\in\mathrm{End}(V)$  is the
adjoint of $\phi_{X}\in\mathrm{End}(V)$ with respect to $g$.
Assume, moreover, that there is a vector field $e$ on $M$ such
that $\phi_{e} = - \mathrm{Id}_{V}$ (where $\mathrm{Id}_{V}$ is
the identity endomorphism of $V$). Let  $\omega\in\Gamma (V)$
(usually called a primitive section) be $\nabla$-parallel such
that the map
\begin{equation}\label{psi-omega}
\psi^{\omega} : TM\rightarrow V,\quad \psi^{\omega} (X) = -
\phi_{X}(\omega )
\end{equation}
is an isomorphism.  Define  a multiplication $\circ$ on $TM$,
with unit field $e$, by
$$
X\circ Y = (\psi^{\omega})^{-1}\left( \phi_{X}\phi_{Y}\omega
\right) .
$$
Using $\psi^{\omega}$, we may transport the metric $g$ and the
connection $\nabla$ to  a metric $g^{\omega}$ and a connection
$\nabla^{\omega}$ on $TM.$  It is easy to see that $\circ$ is
independent of  the choice of primitive section,  while for
$g^{\omega}$ and $\nabla^{\omega}$ the choice of $\omega$ is
essential. As proved by K. Saito \cite{saito}, $(M,\circ ,
e,g^{\omega})$ is a Frobenius $\mathbb{K}$-manifold, with
Levi-Civita connection $\nabla^{\omega}.$

\begin{rem}\label{saito-structure}{\rm In examples coming from singularity theory,
the bundle $V$ comes equipped with two additional
endomorphisms, $R_{0}$ and $R_{\infty}$, satisfying the
conditions:
\begin{align*}
&\nabla R_{0}+ \phi = [\phi , R_{\infty}], \quad [R_{0},
\phi ]=0,\quad R_{0}^{*}= R_{0};\\
&\nabla R_{\infty}=0,\quad R_{\infty}^{*} + R_{\infty} = -w\mathrm{Id}_{V},
\end{align*}
where $w\in\mathbb{K}.$
If $\omega$ is homogeneous, i.e.
$R_{\infty}(\omega ) = - q\omega$ for $q\in\mathbb{K}$, then
$$
E^{\omega}:= (\psi^{\omega})^{-1}(R_{0}(\omega ))
$$
is Euler for $(M,\circ , e, g^{\omega})$, with
$L_{E^{\omega}}(g^{\omega}) = \left( 2(1+q) -w\right) g^{\omega}$,
and
$$
R_{\infty}^{\omega}
:= (\psi^{\omega})^{-1}\circ R_{\infty}\circ \psi^{\omega}=
\nabla^{\omega} (E^{\omega}) - (1+q) \mathrm{Id}.
$$
The data $(\nabla ,
\phi , g, R_{0}, R_{\infty})$  is usually
called a Saito structure (of weight $w$) on $V$. }

\end{rem}

\subsection{Adjunction of a variable to a Frobenius manifold}\label{adjunction}

Following \cite{sabbah}, we now recall that if $(M, \circ , e, g)$
is a Frobenius $\mathbb{K}$-manifold, then $M\times \mathbb{K}$ is
also Frobenius, with unit field $\frac{\partial}{\partial \tau}$
(where $\tau$ is the coordinate on $\mathbb{K}$). In order to
explain this statement, consider the direct sum bundle $TM\oplus
L$, where
\begin{equation}\label{tm-l}
\pi :L=M\times\mathbb{K}\rightarrow M
\end{equation}
is the trivial rank one bundle. Fix $v\in \mathbb{K}$ non-zero
(seen as a constant section of $L$) and define a metric
$\tilde{g}$ on the bundle $L$ by $\tilde{g}(v,v)=1.$ Let $D$ be
the standard trivial connection on $L$ and $\nabla$ the Levi-Civita
connection of $g$. Consider the Higgs field
$$
\phi_{X}(Y):= - X\circ Y,\quad X,Y\in TM,
$$
trivially extended, as a $1$-form on $M$ with values in
$\mathrm{End}(TM\oplus L).$ It is easy to check that the data
 \begin{equation}\label{tm-L}
(\nabla^{L} :=  \nabla\oplus D,\quad g^{L} := g\oplus
\tilde{g},\quad \phi^{L} :=\phi  )
\end{equation}
on  $TM\oplus L$ satisfies relations (\ref{saito0}), and so does the data
\begin{equation}\label{data-first}
\left( \pi^{*}(\nabla^{L}),\quad \pi^{*}(g^{L}),\quad \phi^{\prime}:=
\pi^{*}(\phi^{L}) - d\tau \otimes \mathrm{Id}_{\pi^{*}(TM\oplus L)}\right)
\end{equation}
on $\pi^{*}(TM\oplus L)$.
Any Legendre field $X_{0}$ on $(M, \circ , e, g)$ (i.e. a parallel invertible
vector field) determines a primitive section
$\omega :=\pi^{*}(X_{0}+ v)$ for
$(\pi^{*}(TM\oplus L), \pi^{*}(\nabla^{L}), \pi^{*}(g^{L}),\phi^{\prime})$.
The isomorphism
$$
\psi^{\omega }: T(M\times \mathbb{K})\rightarrow \pi^{*}(TM\oplus
L) ,\quad \psi^{\omega}(Z) = -\phi^{\prime}_{Z}(\omega )
$$
is given by
\begin{equation}\label{psi-omega}
\psi^{\omega}(X) = \pi^{*}(X\circ X_{0}),
\quad \psi^{\omega}\left(\frac{\partial}{\partial \tau}\right) =\pi^{*}( X_{0}+v) ,\quad X\in TM
\end{equation}
and the induced Frobenius structure
$(\circ^{(1)}, \frac{\partial}{\partial \tau}, g^{(1)})$
on $M\times \mathbb{K}$ can be described  as follows:
\begin{equation}\label{r-1}
 X \circ^{(1)}Y = X\circ Y, \quad X\circ^{(1)}
\frac{\partial}{\partial \tau} =
\frac{\partial}{\partial \tau}\circ^{(1)} X=
X, \quad
\frac{\partial}{\partial\tau} \circ^{(1)}\frac{\partial}{\partial\tau}=
\frac{\partial}{\partial\tau}
\end{equation}
for any $X,Y\in TM$ and
\begin{equation}\label{r-2}
g ^{(1)}= g^{X_{0}} + d\tau \otimes i_{e}(g^{X_{0}}) +
i_{e}(g^{X_{0}}) \otimes d\tau +
(g^{X_{0}}(e,e) +1) d\tau \otimes d\tau ,
\end{equation}
where
$$
g^{X_{0}}(X,Y) := g(X_{0}\circ X, X_{0}\circ Y)
$$
is the Legendre transformation of $g$ by the Legendre field
$X_{0}$. The Frobenius structure on $M\times\mathbb{K}$
constructed in this way, with $X_{0}=e$, is usually called the
Frobenius structure obtained from $(M,\circ ,e,g)$ by adding a
variable (\cite{sabbah}, Chapter VII).

\begin{rem}\label{comutare} {\rm
If $X_{0}$ is a Legendre field on a Frobenius
manifold $(M,\circ , e, g)$, then
$(M, \circ , e, g^{X_{0}})$ is also Frobenius (see e.g. \cite{manin}).
From (\ref{r-1}) and (\ref{r-2}) we deduce that
the Frobenius structure on
$M\times\mathbb{K}$, obtained as above by using $\pi^{*}(X_{0}+v)$ as primitive
section, coincides with the Frobenius structure obtained
by adding a variable to the Legendre transformation
$(M,\circ ,e, g^{X_{0}})$ of $(M,\circ ,e,g).$}
\end{rem}

\begin{rem}\label{add-var-saito}{\rm Assume now that the Frobenius manifold
$(M, \circ , e, g)$ has an Euler field $E$ such that
$L_{E}(g) = dg$, for $d\in \mathbb{K}$, and let
$$
\left(\nabla ,\quad g,\quad \phi ,\quad R_{0}:= -\phi_{E},\quad R_{\infty}:=
\nabla E -\frac{1}{2}\left( w+d\right)\mathrm{Id}_{TM}\right)
$$
be the associated Saito structure on $TM$ of weight $w$, where
$\nabla$ is the Levi-Civita connection of $g$ and
$\phi_{X}(Y):= - X\circ Y$ is the
Higgs field of the Frobenius manifold.
Consider $R_{0}$ and $R_{\infty}$ as endomorphisms
of $TM\oplus L$, acting trivially on $L$.
Then
\begin{equation}\label{data}
\left( \pi^{*}(\nabla^{L}),\quad \pi^{*}(g^{L}),\quad \phi^{\prime},\quad
R^{\prime}_{\infty}, \quad R^{\prime}_{0}\right) ,
\end{equation}
where $\pi$ is the projection (\ref{tm-l}),
$\nabla^{L}$, $g^{L}$ and $\phi^{\prime}$ are defined as before by (\ref{tm-L}), (\ref{data-first}), and
$$
R^{\prime}_{\infty}:= \pi^{*}\left( R_{\infty} -
\frac{w}{2}\mathrm{Id}_{V}\right) ,\quad
R^{\prime}_{0}:= \pi^{*}(R_{0})
+\tau \mathrm{Id}_{\pi^{*}(TM\oplus L)},
$$
is a Saito structure of weight $w$ on $\pi^{*}( TM\oplus L)$.
Remark that $\pi^{*}(e + v)$ is $\pi^{*}(\nabla^{L})$-flat, and,
moreover, it is an eigenvector of $R^{\prime}_{\infty}$ if and
only if $d=2.$ Therefore, if $d=2$, $\omega := \pi^{*}(e + v)$ is
a primitive homogeneous section for the Saito bundle $\pi^{*}(
TM\oplus L)$ with data (\ref{data}). The induced Frobenius
structure on $M\times\mathbb{K}$ has an Euler field, namely $E+
R$, where $R_{(p,\tau )}:= \tau\frac{\partial}{\partial \tau}$ is
the radial field. This is because
$$
E+ R = (\psi^{\omega})^{-1} R^{\prime}_{0} \left( \pi^{*}(e+v)\right) ,
$$
see  Remark \ref{saito-structure} and
(\ref{psi-omega}) (with $X_{0}=e$).}
\end{rem}

\subsection{Iterations}
\label{iterations-section}

In this section we determine the general form of the Frobenius structure
obtained by Legendre transformations and successively adding variables to a Frobenius manifold, as follows.

\begin{prop}\label{iterations}
The Frobenius structure on $M\times \mathbb{K}^{r}$  ($r\geq 1$)
obtained from a Frobenius $\mathbb{K}$-manifold $(M,\circ ,e,g)$
by Legendre transformations and adding variables has
multiplication $\circ^{(r)}$ given by
\begin{equation}\label{multi-m}
X \circ^{(r)} Y = X\circ Y, \quad X
\circ^{(r)}\frac{\partial}{\partial \tau_{i}}= X, \quad
\frac{\partial}{\partial \tau_{i}} \circ^{(r)}
\frac{\partial}{\partial \tau_{j}} = \frac{\partial}{\partial
\tau_{\mathrm{min}\{ i,j\}}}
\end{equation}
for any $X, Y\in TM$, $1\leq i,j\leq r$, unit field
$\frac{\partial}{\partial \tau^{r}}$ and metric
\begin{equation}\label{g-k}
g^{(r)} = g^{Z_{0}} +\sum_{k=1}^{r}\left( d\tau_{k}\otimes
i_{e}(g^{Z_{0}} )+ i_{e}(g^{Z_{0}}) \otimes d\tau_{k}
\right)+\sum_{i,j=1}^{r} g^{(r)}_{ij} d\tau_{i} \otimes d\tau_{j},
\end{equation}
where $Z_{0}$ is a Legendre field on $(M, \circ, e, g)$ and
$g_{ij}^{(r)}$ are constants, satisfying
\begin{equation}\label{relation-coef}
g^{(r)}_{ij} = g^{(r)}_{ji} = g^{(r)}_{\mathrm{min}\{ i,j\} ,
r},\quad \forall 1\leq i,j\leq r.
\end{equation}
Above  $(\tau_{1},\cdots , \tau_{r})$ is the standard coordinate system of $\mathbb{K}^{r}.$
\end{prop}

\begin{proof}
We only prove the statement for $r=2$, the general case being
similar. Let $X_{0}$ be a Legendre field on $(M, \circ  ,
e, g)$ and $(M\times\mathbb{K},\circ^{(1)},
\frac{\partial}{\partial \tau_{1}},g^{(1)})$ the Frobenius
manifold obtained
by adding a variable to the Legendre transformation $(M,\circ , e, g^{X_{0}})$
of $(M, \circ ,e,g).$  Thus,
from relations (\ref{r-1}) and (\ref{r-2}) of Section
\ref{adjunction},
$\circ^{(1)}$ and
$g^{(1)}$ are given by
\begin{equation}\label{rr-1}
 X \circ^{(1)}Y = X\circ Y, \quad X\circ^{(1)}
\frac{\partial}{\partial \tau_{1}} =
\frac{\partial}{\partial \tau_{1}}\circ^{(1)} X=
X, \quad
\frac{\partial}{\partial\tau_{1}} \circ^{(1)}\frac{\partial}{\partial\tau_{1}}=
\frac{\partial}{\partial\tau_{1}}
\end{equation}
for any $X,Y\in TM$ and
\begin{equation}\label{rr-2}
g ^{(1)}= g^{X_{0}} + d\tau_{1} \otimes i_{e}(g^{X_{0}}) +
i_{e}(g^{X_{0}}) \otimes d\tau_{1} +
(g^{X_{0}}(e,e) +1) d\tau_{1} \otimes d\tau_{1}.
\end{equation}
Let $(M\times\mathbb{K}^{2},\circ^{(2)}, \frac{\partial}{\partial
\tau_{2}}, g^{(2)})$ be the Frobenius manifold obtained by adding
a variable to $(M\times \mathbb{K}, \circ^{(1)},
\frac{\partial}{\partial \tau_{1}}, (g^{(1)})^{Z})$, where $Z$ is
a Legendre field on $(M\times \mathbb{K}, \circ^{(1)},
\frac{\partial}{\partial \tau_{1}}, g^{(1)})$. We need to check
that $\circ^{(2)}$ and $g^{(2)}$ are given by (\ref{multi-m}) and
(\ref{g-k}) (with $r=2$), for a Legendre field $Z_{0}$ on $(M,
\circ , e, g)$ which needs to be determined. It is easy to check
the statement for $\circ^{(2)}$. We only prove the statement for
$g^{(2)}.$ In order to prove that $g^{(2)}$ is of the required
form, we first notice that $Z$, being a Legendre field on
$(M\times\mathbb{K}, \circ^{(1)}, \frac{\partial}{\partial
\tau^{1}},g^{(1)})$, it is parallel with respect to the
Levi-Civita connection of $g^{(1)}$ and therefore it decomposes
into a sum
$$
Z = Z^{TM} + c \frac{\partial}{\partial \tau^{1}}
$$
where $Z^{TM}$ is a vector field on $M$ which is parallel with
respect to the Levi-Civita connection of $g^{X_{0}}$ and $c\in
\mathbb{K}$  (since $g^{(1)}$ is given by (\ref{rr-2}), any
parallel vector field on $(M\times \mathbb{K}, g^{(1)})$ is a sum
of a parallel vector field on $(M, g^{X_{0}})$ and  a constant
multiple of $\frac{\partial}{\partial\tau^{1}}$ -  this can be
checked directly; see also Remark \ref{variable-ajut}).

With this preliminary remark, we now
claim that $g^{(2)}$ is given by
(\ref{g-k}), with
\begin{equation}\label{zo}
Z_{0}:= X_{0}\circ (Z^{TM} + ce),
\end{equation}
and that $Z_{0}$ is a Legendre field on $(M, \circ , e, g).$ We
divide the proof of this claim in two steps: first, we prove that
$Z_{0}$ is Legendre on $(M, \circ , e, g)$; next, we prove that
$g^{(2)}$ is of the  form (\ref{g-k}), with $Z_{0}$ given by
(\ref{zo}).

To prove that $Z_{0}$ given by (\ref{zo}) is a Legendre field on $(M, \circ ,e,g)$,
we  recall that the Levi-Civita connections of $g$ and $g^{X_{0}}$
are related by a Legendre transformation
$$
\nabla^{g^{X_{0}}} = X_{0}^{-1}\circ \nabla^{g}\circ X_{0}.
$$
Therefore, since $Z^{TM}$ is $\nabla^{g^{X_{0}}}$-parallel, so is
$Z^{TM} + ce$ (because $\nabla^{g} X_{0}=0$) and $Z_{0}$, defined
by (\ref{zo}), is parallel with respect to $\nabla^{g}$. Moreover,
it may be checked that if $W\in {\mathcal X}(M\times\mathbb{K})$
is the inverse of $Z$ with respect to $\circ^{(1)}$ (which exists,
because $Z$ is Legendre on $(M\times\mathbb{K}, \circ^{(1)},
\frac{\partial}{\partial\tau_{1}}, g^{(1)})$), then the projection
$W^{TM}$ of $W$ to $TM$ is a vector field on $M$ and
$$
\left( W^{TM}+ \frac{1}{c}e\right)\circ \left( Z^{TM}+ c e\right) =e,
$$
(note also that $c\neq 0$),  i.e. $Z^{TM}+ce$, hence also $Z_{0}$,
are invertible with respect to $\circ .$ We conclude that $Z_{0}$
is a Legendre field on $(M,\circ ,e,g)$, as required.

It remains to prove that $g^{(2)}$ is given by (\ref{zo}).
For this, we use  relations (\ref{r-1}), (\ref{r-2}) and that
$\frac{\partial}{\partial \tau^{1}}$ is the unit field for
$\circ^{(1)}$ and we get:
\begin{equation}\label{ecu}
g^{(2)} = (g^{(1)})^{Z} + d\tau_{2}\otimes i_{\frac{\partial}{\partial
\tau^{1}}} (g^{(1)})^{Z} +
i_{\frac{\partial}{\partial
\tau^{1}}}(g^{(1)})^{Z}\otimes d\tau_{2} +
\left( g^{(1)}(Z,Z) + 1\right) d\tau_{2}\otimes d\tau_{2}.
\end{equation}
From (\ref{rr-1}), (\ref{rr-2}) and (\ref{ecu}),  for any $X, Y\in TM$,
\begin{equation}\label{g-2-circ}
g^{(2)}(X,Y) = g^{(1)} (Z\circ^{(1)}X, Z\circ^{(1)}Y)
= g^{X_{0}}\left( (Z^{TM} + ce)^{2}\circ X, Y\right)  ,
\end{equation}
because
\begin{equation}\label{z-c}
Z\circ^{(1)} X = \left( Z^{TM} + c\frac{\partial}{\partial \tau^{1}}\right)\circ^{(1)} X= Z^{TM}\circ X + cX
= (Z^{TM} + ce)\circ  X.
\end{equation}
Thus,
$$
g^{(2)}(X,Y) = g(Z_{0}\circ X, Z_{0}\circ Y) = g^{Z_{0}}(X,Y),\quad
\forall X,Y\in TM.
$$
From (\ref{rr-2}), (\ref{ecu}) and (\ref{z-c}),
\begin{align*}
g^{(2)} (\frac{\partial}{\partial \tau^{1}}, Y) & =
(g^{(1)})^{Z}\left( \frac{\partial}{\partial \tau_{1}},Y\right)=
g^{(1)}(Z\circ^{(1)}\frac{\partial}{\partial \tau^{1}},
Z\circ^{(1)} Y)\\
& = g^{(1)} ( Z, (Z^{TM}+ ce)\circ Y ) = g^{(1)} ( Z^{TM} +
c\frac{\partial}{\partial \tau_{1}}, (Z^{TM}+ ce)\circ Y)\\
&= g^{X_{0}} \left( (Z^{TM} + ce)^{2}, Y\right) = g^{Z_{0}}(e,Y)
\end{align*}
where in the second line we used (\ref{z-c}) and that $\frac{\partial}{\partial
\tau^{1}}$ is the unit field for $\circ^{(1)}$, and in the last line we used the definition (\ref{rr-2}) of
$g^{(1)}.$ The above computation also implies
$$
g^{(2)} (\frac{\partial}{\partial \tau^{2}}, Y)=
(g^{(1)})^{Z}
\left(\frac{\partial}{\partial \tau_{1}}, Y\right)=
g^{Z_{0}}(e,Y), \quad\forall Y\in TM.
$$
We proved that $g^{(2)}$ is of the form (\ref{g-k}), with $r=2$,
$Z_{0}$ given by (\ref{zo}) and coefficients
$$
g^{(2)}_{ij}:= g^{(2)} \left(\frac{\partial}{\partial \tau^{i}},
\frac{\partial}{\partial \tau^{j}}\right) = g^{(2)} \left(
\frac{\partial}{\partial \tau^{i}}\circ^{(2)}\frac{\partial}{\partial
\tau^{j}},\frac{\partial}{\partial\tau^{r}}\right) =
g^{(2)}_{\mathrm{min}\{i,j\} r} ,\quad\forall 1\leq i, j\leq 2,
$$
because $\frac{\partial}{\partial \tau_{i}}\circ^{(2)}
\frac{\partial}{\partial \tau_{j}} = \frac{\partial}{\partial
\tau_{\mathrm{min} \{i,j\}}}.$ Remark that $g^{(2)}_{ij}$ are
constant, because
\begin{align*}
g^{(2)}\left(\frac{\partial}{\partial\tau^{1}},
\frac{\partial}{\partial\tau^{1}}\right) &=
g^{(2)}\left( \frac{\partial}{\partial\tau^{1}},
\frac{\partial}{\partial\tau^{2}}\right)
= g^{(1)}( Z,Z)\\
g^{(2)}\left(\frac{\partial}{\partial\tau^{2}},
\frac{\partial}{\partial\tau^{2}}\right)
&= g^{(1)}(Z,Z)+1
\end{align*}
and $Z$ is Legendre (hence, parallel) on $(M\times\mathbb{K},
\circ^{(1)}, \frac{\partial}{\partial\tau^{1}}, g^{(1)})$.

\end{proof}

In the following sections we will search for more general
Frobenius structures on total spaces of vector bundles.

\section{Total spaces of vector
bundles as $F$-manifolds}\label{four1}

Here and in the following sections we fix a $\mathbb{K}$-vector bundle
$\pi : V\rightarrow M$ with the following additional data:\\

$\bullet$ a {\bf connection $D$ on $V$}, which induces a
decomposition
\begin{equation}\label{m-c}
T_{v}V = T_{p}M\oplus V_{p}, \quad\forall v\in V
\end{equation}
into horizontal and vertical subspaces. The horizontal lift of a
vector field $X\in {\mathcal X}(M)$ will be denoted $\tilde{X}$.
When a longer term $Expression$ is lifted to $V$, the lift will be
denoted by $\left[ Expression \right]^{\widetilde{}}.$ Often
sections of $V$ will be considered (without mentioning explicitly)
as vertical vector fields on the manifold $V$. Recall that
\begin{equation}\label{curv}
[\tilde{X}, \tilde{Y}]_{v} = {[X,Y]}^{\widetilde{}}_{v}- R^{D}_{X,
Y}v, \quad \forall X, Y\in {\mathcal X}(M), \quad\forall v\in V
\end{equation}
and
\begin{equation}\label{vect-curb}
[\tilde{X}, s] = D_{X}s, \quad\forall X\in {\mathcal X}(M), \quad
\forall s\in \Gamma (V).
\end{equation}

$\bullet$  a (fiber preserving) {\bf multiplication $\circ_{M}$ on
$TM$}, which is  commutative, associative,
with unit $e_{M}\in {\mathcal X}(M)$;\\

$\bullet$ a (fiber preserving) {\bf multiplication $\circ_{V}$ on
the bundle $V$}, which is
commutative, associative, with unit $e_{V}\in\Gamma (V)$;\\

$\bullet$ A {\bf bundle morphism}
$$
\Theta : TM\oplus V \rightarrow TM, \quad (X, v)\rightarrow \Theta
(X,v),
$$
such that
\begin{equation}\label{Theta}
\Theta (X,e_{V}) = X,\quad\forall X\in TM.
\end{equation}

We construct from the data
$(D,\circ_{M}, \circ_{V}, \Theta )$
a (fiber preserving) multiplication $\circ$ on $TV$, as follows:
$$
\tilde{X}\circ\tilde{Y} := \left[
X\circ_{M}Y\right]^{\widetilde{}}, \quad v_{1}\circ v_{2} :=
v_{1}\circ_{V}v_{2},\quad v\circ \tilde{X} = \tilde{X}\circ v :=
\left[ \Theta (X,v)\right]^{\widetilde{}}
$$
where $X,Y\in T_{p}M$ and  $v_{1}, v_{2}\in V_{p}$.
From (\ref{Theta}),
$e_{V}$ is the unit field for $\circ .$
We also require that $\circ$ is associative and commutative, which
is equivalent to
$$
\Theta (X,v) = X\circ_{M} \alpha (v), \quad\forall X\in TM,
\quad\forall v\in V,
$$
with
$$
\alpha : V\rightarrow TM,
\quad \alpha (v):= \Theta (e_{M},v)
$$
satisfying
\begin{equation}\label{v1-v2}
\alpha (v_{1}\circ_{V} v_{2}) = \alpha (v_{1})\circ_{M}\alpha (v_{2}),
\quad \alpha (e_{V}) = e_{M}.
\end{equation}
Thus, $\circ$ is given by
\begin{equation}\label{circ}
\tilde{X}\circ\tilde{Y} :=
\left[{X\circ_{M}Y}\right]^{\widetilde{}}, \quad v_{1}\circ v_{2}
:= v_{1}\circ_{V}v_{2},\quad v\circ \tilde{X} := \tilde{X}\circ v
= \left[ {\alpha (v)\circ_{M}X}\right]^{\widetilde{}}.
\end{equation}

Our main result from this section is the following.

\begin{prop}\label{multiplication} The multiplication $\circ$ defines
an $F$-manifold
structure on $V$ (with unit field $e_{V}$) if and only if the following conditions are
satisfied:\\

1) $(M, \circ_{M}, e_{M})$ is an $F$-manifold;\\

2) the connection $D$ is flat;\\

3) For any $D$-parallel (local) section $s\in \Gamma (V)$,
 \begin{equation}\label{alpha}
L_{\alpha (s)} (\circ_{M} ) =0.
 \end{equation}

4) If $s_{1},s_{2}\in\Gamma (V)$ are $D$-parallel, then so is
$s_{1}\circ_{V} s_{2}$ and, moreover,
\begin{equation}\label{croset}
[\alpha (s_{1}) , \alpha (s_{2})] =0.
\end{equation}
\end{prop}

\begin{proof}
We need to show that
\begin{equation}\label{F-man-ext}
L_{{\mathcal W}_{1}\circ {\mathcal W}_{2}} (\circ )({\mathcal
W}_{3}, {\mathcal W}_{4}) = {\mathcal W}_{1}\circ L_{\mathcal
W_{2}}(\circ ) (\mathcal W_{3}, \mathcal W_{4}) + \mathcal
W_{2}\circ L_{\mathcal W_{1}}(\circ ) (\mathcal W_{3}, \mathcal
W_{4})
\end{equation}
for any vector fields $\mathcal W_{1}, \mathcal W_{2}, \mathcal
W_{3}, \mathcal W_{4}\in {\mathcal X}(V)$, is equivalent to the
conditions from the statement of the proposition. Suppose first
that (\ref{F-man-ext}) holds and let $\mathcal W_{1} = \tilde{X}$,
$\mathcal W_{2} = \tilde{Y}$, $\mathcal W_{3}=\tilde{Z}$,
$\mathcal W_{4} = \tilde{T}$ be $D$-horizontal lifts, where
$X,Y,Z,T\in {\mathcal X}(M)$. Using (\ref{curv}) and the
definition of $\circ$, one can show that
$$
L_{\tilde{X}\circ\tilde{Y}} (\circ )(\tilde{Z}, \tilde{T}) =
\tilde{X}\circ L_{\tilde{Y}}(\circ )(\tilde{Z}, \tilde{T}) +
\tilde{Y}\circ L_{\tilde{X}} (\circ )(\tilde{Z},\tilde{T})
$$
if and only if
$$
R^{D}_{X\circ Y, Z\circ T} = 0
$$
and
$$
L_{X\circ Y}(\circ_{M})(Z, T)  = X\circ L_{Y}(\circ_{M})(Z,T)
+ Y\circ L_{X}(\circ_{M}) (Z, T),
$$
i.e. $D$ is flat and $(M, \circ_{M}, e_{M})$ is an $F$-manifold.
Thus, the first two conditions from the proposition hold. Next,
let $\mathcal W_{1}:=\tilde{X}$,  $\mathcal W_{2}= s\in\Gamma
(V)$, $\mathcal W_{3}= \tilde{Z}$ and $\mathcal W_{4}= \tilde{T}$.
Using (\ref{curv}) (with $R^{D}=0$) and (\ref{vect-curb}), one can
check that
$$
L_{\tilde{X}\circ s}(\circ )(\tilde{Z}, \tilde{T}) =
\tilde{X}\circ L_{s}(\circ )(\tilde{Z}, \tilde{T}) + s\circ
L_{\tilde{X}}(\circ )(\tilde{Z}, \tilde{T})
$$
if and only if
\begin{equation}\label{alpha-s}
L_{\alpha (s)} (\circ_{M}) (Z, T) = - \alpha (D_{Z\circ_{M}T}s) +
T\circ_{M} \alpha (D_{Z}s) +
Z\circ_{M}\alpha (D_{T}s),
\end{equation}
which is equivalent to (\ref{alpha}) (using that $D$ is flat).
Thus, the third condition holds as well. Consider now $\mathcal
W_{1}= s_{1}$, $\mathcal W_{2}= s_{2}$, $\mathcal W_{3}=\tilde{Z}$
and $\mathcal W_{4}= \tilde{T}.$ Using (\ref{vect-curb}) again,
one can show that
$$
L_{s_{1}\circ s_{2}}(\circ ) (\tilde{Z},\tilde{T}) = s_{1}\circ
L_{s_{2}}(\circ )(\tilde{Z},\tilde{T}) + s_{2}\circ
L_{s_{1}}(\circ )(\tilde{Z},\tilde{T})
$$
if and only if
\begin{equation}\label{D-V}
D_{Z}(s_{1}\circ_{V}s_{2}) = (D_{Z}s_{1})\circ_{V} s_{2} +
s_{1}\circ_{V} (D_{Z}s_{2}),
\end{equation}
which is equivalent to $D(s_{1}\circ_{V} s_{2}) =0$, when
$D(s_{i})=0$, $i={1,2}$ (again, because $D$ is flat). Finally, let
$\mathcal W_{1}:= \tilde{X}$, $\mathcal W_{2}= s_{2}$, $\mathcal
W_{3} =\tilde{Z}$ and $\mathcal W_{4}= s_{4}$. Using that $(M,
\circ_{M}, e_{M})$ is an $F$-manifold and (\ref{alpha-s}),  it can
be shown that
$$
L_{\tilde{X}\circ s_{2}} (\circ ) (\tilde{Z},s_{4}) =
\tilde{X}\circ L_{s_{2}} (\circ ) (\tilde{Z}, s_{4}) + s_{2}\circ
L_{\tilde{X}}(\circ )( \tilde{X}, s_{4})
$$
if and only if
$$
[\alpha (s_{2}), \alpha (s_{4})] = \alpha \left( D_{\alpha
(s_{2})}s_{4}- D_{\alpha (s_{4})}s_{2}\right) ,
$$
i.e. the fourth condition holds as well. We have proved that if
$(V, \circ , e_{V})$ is an $F$-manifold, then all conditions from
the proposition are satisfied. Conversely, assume that all
conditions are satisfied. Our previous argument shows that
(\ref{F-man-ext}) holds when all $\mathcal W_{i}$ are horizontal,
or when one is vertical and the other three are horizontal, or
when two are vertical and two are horizontal. It is easy to check
that (\ref{F-man-ext}) holds also with the remaining type of
arguments (i.e. three vertical and one horizontal, or all four
vertical). Our claim follows.
\end{proof}

\begin{rem}\label{rem-ss}{\rm  When $(M, \circ_{M}, e_{M})$ is semisimple, i.e. there is a coordinate system
$(u^{1}, \cdots , u^{n})$ (called canonical) such that
$\frac{\partial}{\partial u^{i}}\circ_{M}\frac{\partial}{\partial
u^{j}}= \delta_{ij}\frac{\partial}{\partial u^{i}}$, for any $i,
j$, condition (\ref{alpha}) implies condition (\ref{croset}). The
reason is that a vector field $X$ on a semisimple $F$-manifold
$(M, \circ_{M}, e_{M})$ satisfies $L_{X}(\circ_{M})=0$ if and only
if it is of the form $X = \sum_{k=1}^{n}
a_{k}\frac{\partial}{\partial u^{k}}$, for some constants $a_{k}$
(see e.g. \cite{manin}). Any two such vector fields commute.}
\end{rem}

\section{Admissible metrics on $(V, \circ , e_{V})$}\label{admissible}

We consider the setting of Section  \ref{four1} and we assume that
all conditions from Proposition \ref{multiplication} are
satisfied, i.e. $(V, \circ ,e_{V})$ is an $F$-manifold. We now add
to the picture a metric $g_{M}$ on $M$, invariant with respect to
$\circ_{M}$, and a metric $g_{V}$ on the bundle $V$, invariant
with respect to $\circ_{V}.$ In the remaining part of the paper we
assume that the $(2,0)$-tensor field on the manifold $V$, defined
by
\begin{equation}\label{metric}
g(\tilde{X}, \tilde{Y}) := g_{M}(X, Y), \quad g(v_{1}, v_{2}) :=
g_{V}(v_{1}, v_{2}), \quad g(v, \tilde{X}) := g_{M} (\alpha (v),
X),
\end{equation}
for any $X,Y\in TM$ and $v,v_{1},v_{2}\in V$, is non-degenerate.
This is equivalent to the non-degeneracy of the metric
$g_{V}-\alpha^{*}g_{M}$ of the bundle $V$ (easy check). Also, from
(\ref{v1-v2}) and the invariance of $g_{M}$ and $g_{V}$, the
metric $g$ is invariant on $(V,\circ ,e_{V})$.

\begin{prop}\label{admissible-cond}
Assume that $(V, \circ , e_{V})$ is an $F$-manifold.
Then the metric $g$  defined by
(\ref{metric})  is
admissible on $(V, \circ , e_{V})$ if and only if $D(g_{V}) =0$ and the coidentity
$\epsilon_{M}:= g_{M}(e_{M}, \cdot )$ is closed.
\end{prop}

\begin{proof} We first remark that the unit section $e_{V}\in\Gamma (V)$
is $D$-parallel: this follows by  taking the covariant derivative
(with respect to $D$) of the equality $e_{V}\circ_{V} e_{V}
=e_{V}$ and using relation (\ref{D-V}). Thus,
\begin{equation}\label{ev-section}
[\tilde{X}, e_{V}]= D_{X}(e_{V})=0,\quad \forall X\in {\mathcal
X}(M).
\end{equation}
The Koszul formula for the Levi-Civita connection $\nabla$ of $g$,
together with (\ref{ev-section}), $\alpha (e_{V})= e_{M}$ and
$R^{D}=0$, imply that for any $Y, Z\in {\mathcal X}(M)$,
$$
g(\nabla_{\tilde{Y}}(e_{V}), \tilde{Z}) = \frac{1}{2} \left(
g_{M}(\nabla^{M}_{Y} (\alpha (e_{V})), Z) - g_{M}(
\nabla^{M}_{Z}(\alpha (e_{V})), Y)\right) =\frac{1}{2}
(d\epsilon_{M})(Y,Z),
$$
where $\nabla^{M}$ is the Levi-Civita connection of $g_{M}$.
Similarly,
$$
g(\nabla_{\tilde{Y}}(e_{V}), s) = - g(\nabla_{s}(e_{V}),
\tilde{Y})=\frac{1}{2}D_{Y}(\epsilon_{V})(s),\quad
g(\nabla_{s}e_{V}, \tilde{s})= 0,
$$
for any $s, \tilde{s}\in\Gamma (V)$, where
$\epsilon_{V}\in \Gamma (V^{*})$ is the $g_{V}$-dual to $e_{V}$, i.e.
$$
\epsilon_{V}(v) := g_{V}(e_{V}, v),\quad \forall v\in V.
$$
Thus, $\nabla (e_{V} )=0$ is equivalent to $D(\epsilon_{V}) =0$
and $d\epsilon_{M}=0.$ Now, we claim that $D(\epsilon_{V}) = 0$ is
equivalent to $D(g_{V})=0.$ This is a consequence of the relation
$$
g_{V}(s_{1},s_{2}) = \epsilon_{V}(s_{1}\circ_{V} s_{2}),
\quad\forall s_{1},s_{2}\in \Gamma ( V)
$$
and the fact that if $s_{1}$ and $s_{2}$ are $D$-parallel, then so
is $s_{1}\circ_{V} s_{2}$ (from  Proposition
\ref{multiplication}). Our claim follows.

\end{proof}

\section{Frobenius structures on $V$}\label{frobenius}

We consider the almost Frobenius structure $(\circ , e_{V}, g)$ on
$V$ constructed from the data $(D,\circ_{M},
e_{M},\circ_{V},e_{V}, g_{M},g_{V})$ as in the previous section,
and we assume that $(V,\circ , e_{V}, g)$ is an $F$-manifold with
admissible metric (hence the conditions from Propositions
\ref{multiplication} and \ref{admissible-cond} are satisfied) and
the base $(M,\circ_{M},e_{M},g_{M})$ is a Frobenius manifold. In
this section we compute the curvature of $g$ and we find
conditions on the morphism $\alpha$ which insure that $g$ is flat,
or $(V, \circ ,e_{V}, g)$ is Frobenius. Our main result is Theorem
\ref{corectat}, which describes all real Frobenius structures
$(\circ , e_{V},g)$ on $V$, with positive definite metric $g$,
when $M$ is (real) semisimple with non-vanishing rotation coefficients.\\

We begin by computing the Levi-Civita connection of $g$.

\begin{lem}\label{levi-civita}
The Levi-Civita connection $\nabla$ of $g$ is given by
\begin{align*}
&g(\nabla_{\tilde{Y}}\tilde{X}, \tilde{Z}) = g_{M}(
\nabla^{M}_{Y}X,
Z)\\
&g(\nabla_{\tilde{Y}}\tilde{X}, s)= g_{M} (\nabla_{Y}^{M}X, \alpha
(s))+ \frac{1}{2} \left( g_{M}(\nabla^{M,D}_{X}
(\alpha )(s),Y) + g_{M} (\nabla^{M,D}_{Y}(\alpha )(s),X)\right)\\
&g(\nabla_{s_{1}}\tilde{X}, s_{2}) =0\\
&g(\nabla_{s}\tilde{X}, \tilde{Y}) = \frac{1}{2}\left(
g_{M}(\nabla^{M,D}_{X}(\alpha )(s), Y) - g_{M} (\nabla^{M,D}_{Y}
(\alpha )(s), X)
\right)\\
&\nabla_{\tilde{X}} s= \nabla_{s}\tilde{X} + D_{X}s\\
&\nabla_{s_{1}}s_{2}=0,
\end{align*}
for any vector fields $X, Y, Z\in {\mathcal X}(M)$ and sections
$s, s_{1}, s_{2}\in\Gamma (V)$. Above $\nabla^{M}$ is the
Levi-Civita connection of $g_{M}$, the morphism $\alpha$ is viewed
as a section of $TM\otimes V^{*}$ and $\nabla^{M,D}$ is the
product connection $\nabla^{M}\otimes D$ on $TM\otimes V^{*}.$
\end{lem}

\begin{proof}
The proof is a straightforward computation, which uses the
flatness of $D$ (which implies
$[\tilde{X},\tilde{Y}]={[X,Y]}^{\widetilde{}}$, for any $X,Y\in
{\mathcal X}(M)$) and $D(g_{V})=0$ (from Proposition
\ref{admissible-cond}).

\end{proof}

\begin{rem}\label{variable-ajut}{\rm In the above lemma, assume that $\nabla^{M, D}(\alpha )=0$.
Then a vector field is $\nabla$-parallel if and only if it is of
the form $\tilde{X} + s$, where $X$ is $\nabla^{M}$-parallel and
$s$ is a $D$-parallel section of $V$, viewed as a tangent vertical
vector field on $V$; since $g_{M}$ and $D$ are flat, locally there
is a maximal number of parallel vector fields on $(V, g)$ and $g$
is flat as well. This applies to the particular case when $V$ is
the trivial bundle of rank one and the almost Frobenius structure
on $V$ is given by adding a variable to $(M,
\circ_{M},e_{M},g_{M})$; in this case, $\alpha (e_{V})=e_{M}$ and
thus $\nabla^{M,D}(\alpha )=0$ (because $\nabla^{M}(e_{M})=0$ and
$D(e_{V})=0$). }
\end{rem}

Next, we compute the curvature of $g$. For this, we first
introduce new notations, as follows. For a tangent vector
${\mathcal W}\in T_{v}V$, we denote by ${\mathcal
W}^{\mathrm{hor}}$ and ${\mathcal W}^{\mathrm{vert}}$ its
horizontal and vertical components with respect to the
decomposition
$$
T_{v}V = T_{p}M \oplus V_{p}
$$
induced by the connection $D$. From the definition of $g$, a
tangent vector ${\mathcal W}\in T_{v}V$ belongs to
$(T_{p}M)^{\perp}$ (i.e. is $g$-orthogonal to $T_{p}M\subset
T_{v}V$) if and only if
\begin{equation}\label{componente}
{\mathcal W} = - \alpha ({\mathcal W}^{\mathrm{vert}}) + {\mathcal
W}^{\mathrm{vert}}.
\end{equation}
Similarly, $W\in (V_{p})^{\perp}$
(i.e. is $g$-orthogonal to $V_{p}\subset T_{v}V$)
if and only if
\begin{equation}\label{componente2}
g_{M} ({\mathcal W}^{\mathrm{hor}}, \alpha (w)) + g_{V}({\mathcal
W}^{\mathrm{vert}},w)=0, \quad \forall w\in V_{p}.
\end{equation}

\begin{lem}\label{curvature}
The curvature $R^{\nabla}$ of the metric $g$ has the following expression: for
any $\nabla^{M}$-parallel vector fields $X, Y, Z, T\in {\mathcal
X}(M)$ and $D$-parallel sections $s, s_{1},s_{2}\in\Gamma (V)$,

\begin{align*}
R^{\nabla}(\tilde{X}, \tilde{Y}, \tilde{Z}, \tilde{T}) &=
(g_{V}-\alpha^{*}g_{M})((\nabla_{\tilde{X}}\tilde{Z})^{\mathrm{vert}},
(\nabla_{\tilde{Y}}\tilde{T})^{\mathrm{vert}})\\
&- (g_{V}-\alpha^{*}g_{M})
((\nabla_{\tilde{Y}}\tilde{Z})^{\mathrm{vert}},
(\nabla_{\tilde{X}}
\tilde{T})^{\mathrm{vert}});\\
R^{\nabla}(\tilde{X}, \tilde{Y}, \tilde{Z}, s) &=\frac{1}{2}\left(
g_{M} (\nabla^{M}_{X} [\nabla^{M}_{Z}(\alpha (s))], Y )- g_{M}
(\nabla^{M}_{Y}
[\nabla^{M}_{Z}(\alpha (s))],X)\right)\\
& +\left( g_{V}-\alpha^{*}g_{M}\right) \left(
(\nabla_{\tilde{X}}\tilde{Z})^{\mathrm{vert}} ,
(\nabla_{\tilde{Y}}s)^{\mathrm{vert}}\right)\\
&-\left(g_{V}- \alpha^{*}g_{M}\right)\left( (\nabla_{\tilde{Y}}
\tilde{Z})^{\mathrm{vert}} , (\nabla_{\tilde{X}}s)^{\mathrm{vert}}\right) ;\\
R^{\nabla}(\tilde{X},\tilde{Y},s_{1},s_{2})&= - g_{M}(
(\nabla_{\tilde{Y}}s_{1})^{\mathrm{hor}},
(\nabla_{\tilde{X}}s_{2})^{\mathrm{hor}}) + g_{V}
((\nabla_{\tilde{X}}s_{2})^{\mathrm{vert}},
(\nabla_{\tilde{Y}}s_{1})^{\mathrm{vert}})\\
& + g_{M}((\nabla_{\tilde{X}}s_{1})^{\mathrm{hor}},
(\nabla_{\tilde{Y}}s_{2})^{\mathrm{hor}}) - g_{V} (
(\nabla_{\tilde{Y}}s_{2})^{\mathrm{vert}},
(\nabla_{\tilde{X}}s_{1})^{\mathrm{vert}});\\
R^{\nabla}(\tilde{X},s_{1},\tilde{Y},s_{2})&= g_{V}(
(\nabla_{\tilde{X}}s_{2})^{\mathrm{vert}},
(\nabla_{s_{1}}\tilde{Y})^{\mathrm{vert}}) -g_{M} (
(\nabla_{s_{1}}\tilde{Y})^{\mathrm{hor}},
(\nabla_{\tilde{X}}s_{2})^{\mathrm{hor}})\\
R^{\nabla}(s_{1}, s_{2})(s)&=0.
\end{align*}
\end{lem}

\begin{proof}
We compute $R^{\nabla}(\tilde{X}, \tilde{Y}, \tilde{Z},
\tilde{T})$, where $X, Y, Z, T\in {\mathcal X}(M)$ are
$\nabla^{M}$-parallel, as follows:
\begin{align*}
R^{\nabla}(\tilde{X}, \tilde{Y}, \tilde{Z}, \tilde{T})&=
g\left(\nabla_{\tilde{X}}\nabla_{\tilde{Y}}\tilde{Z} -
\nabla_{\tilde{Y}}\nabla_{\tilde{X}}\tilde{Z} -
\nabla_{[\tilde{X},
\tilde{Y}]}\tilde{Z}, \tilde{T}\right)\\
&=\tilde{X}g\left(\nabla_{\tilde{Y}}\tilde{Z}, \tilde{T}\right) -
g\left(\nabla_{\tilde{Y}}\tilde{Z}, \nabla_{\tilde{X}}\tilde{T}\right)\\
& - \tilde{Y}g\left(\nabla_{\tilde{X}}\tilde{Z}, \tilde{T}\right)
+
g\left( \nabla_{\tilde{X}}\tilde{Z}, \nabla_{\tilde{Y}}\tilde{T}\right)\\
&= - g\left(\nabla_{\tilde{Y}}\tilde{Z},
\nabla_{\tilde{X}}\tilde{T}\right) +
g\left(\nabla_{\tilde{X}}\tilde{Z},\nabla_{\tilde{Y}}\tilde{T}\right),
\end{align*}
where we used that $\nabla_{\tilde{Y}}\tilde{Z}$,
$\nabla_{\tilde{X}}\tilde{Z}$ are $g$-orthogonal to $TM$ (from
Lemma \ref{levi-civita}) and $[\tilde{X}, \tilde{Y}] = 0$, because
$R^{D}=0$ and $[X,Y]=0.$ Thus:
\begin{equation}\label{ajut-r}
R^{\nabla}(\tilde{X}, \tilde{Y}, \tilde{Z}, \tilde{T})= -
g(\nabla_{\tilde{Y}}\tilde{Z}, \nabla_{\tilde{X}}\tilde{T}) +
g(\nabla_{\tilde{X}}\tilde{Z},\nabla_{\tilde{Y}}\tilde{T}).
\end{equation}
On the other hand, since $\nabla_{\tilde{Y}}\tilde{Z}$ and
$\nabla_{\tilde{X}}\tilde{T}$ are $g$-orthogonal to $TM$, from
(\ref{componente}),
\begin{equation}\label{comp-nabla}
(\nabla_{\tilde{Y}}\tilde{Z})^{\mathrm{hor}} = -\alpha
((\nabla_{\tilde{Y}}\tilde{Z})^{\mathrm{vert}})
\end{equation}
and similarly for $\nabla_{\tilde{X}}\tilde{T}.$ We obtain:
\begin{align*}
g(\nabla_{\tilde{Y}}\tilde{Z}, \nabla_{\tilde{X}}\tilde{T})& = g(
\nabla_{\tilde{Y}}\tilde{Z},
(\nabla_{\tilde{X}}\tilde{T})^{\mathrm{hor}}
+ (\nabla_{\tilde{X}}\tilde{T})^{\mathrm{vert}}) \\
&= g\left( (\nabla_{\tilde{Y}}\tilde{Z})^{\mathrm{hor}} +
(\nabla_{\tilde{Y}}\tilde{Z})^{\mathrm{vert}},
(\nabla_{\tilde{X}}\tilde{T})^{\mathrm{vert}}\right)\\
& = g_{M} ( (\nabla_{\tilde{Y}}\tilde{Z})^{\mathrm{hor}}, \alpha (
(\nabla_{\tilde{X}}\tilde{T})^{\mathrm{vert}})) + g_{V} (
(\nabla_{\tilde{Y}}\tilde{Z})^{\mathrm{vert}},
(\nabla_{\tilde{X}}\tilde{T})^{\mathrm{vert}})\\
& = (g_{V} -
\alpha^{*}g_{M})((\nabla_{\tilde{Y}}\tilde{Z})^{\mathrm{vert}},
(\nabla_{\tilde{X}}\tilde{T})^{\mathrm{vert}}),
\end{align*}
where in the last equality we used (\ref{comp-nabla}). Combining
this relation with (\ref{ajut-r}) we obtain $R^{\nabla}(\tilde{X},
\tilde{Y}, \tilde{Z},\tilde{T})$, as required. The remaining
expressions of the curvature can be obtained in a similar way.
\end{proof}

In the following proposition we determine a set of conditions on
the morphism $\alpha$ which are equivalent to the flatness of $g$.

\begin{prop}\label{cond-flatness}
Assume that $(V, \circ,e_{V},g)$ is an $F$-manifold with
admissible metric. Let $\{ v_{1},\cdots , v_{r}\}$ be a local
frame of $V$, orthonormal with respect to $g_{V}-\alpha^{*}g_{M}$,
and $X_{1}, \cdots , X_{n}$ a local frame of $TM$, orthonormal
with respect to $g_{M}.$ Define $\epsilon_{i}:=
(g_{V}-\alpha^{*}g_{M})(v_{i}, v_{i})$ and $\tilde{\epsilon}_{j}:=
g(X_{j}, X_{j})$ (thus, $\epsilon_{i}= \tilde{\epsilon}_{j}=1$ in
the complex case or when both $g_{M}$ and $g_{V}- \alpha^{*}g_{M}$
are positive definite). The metric $g$ is flat (and $(V, \circ ,
e_{V}, g)$ is Frobenius)
if and only if the following conditions hold:\\

1) for any vector fields $X,Y,Z,T\in {\mathcal X}(M)$, the expression
\begin{align*}
S_{1}(X, Y, Z, T) :=&\sum_{i=1}^{r}\epsilon_{i} \left( g_{M}
(\nabla^{M,D}_{Y}(\alpha ) (v_{i}), Z)
+ g_{M} (\nabla^{M,D}_{Z}(\alpha )(v_{i}), Y)\right)\\
&\left( g_{M} (\nabla^{M,D}_{X}(\alpha ) (v_{i}), T)
+ g_{M} (\nabla^{M}_{T}(\alpha )(v_{i}), X)\right)\\
\end{align*}
is symmetric in $X$ and $Y$.\\

2) for any (local) $\nabla^{M}$-parallel vector fields $X,Y,Z\in{\mathcal
X}(M)$ and $D$-parallel section
$s\in \Gamma (V)$, the expression
$$
S_{2}(X,Y,Z, s):=
\sum_{i=1}^{r}\epsilon_{i}
d(\alpha (s)^{\flat})(\alpha (v_{i}), Y)
\left( g_{M}(\nabla^{M,D}_{X}(\alpha )
(v_{i}), Z)
+ g_{M}( \nabla^{M,D}_{Z}(\alpha )(v_{i}), X)\right)
$$
(where the ``$\flat$``  denotes the $g_{M}$-dual $1$-form) satisfies
$$
S_{2}(X, Y, Z, s) -
S_{2}(Y, X, Z, s)
= 2\left( g_{M}(\nabla^{M}_{Y}
\nabla^{M}_{Z}(\alpha (s)), X)
- g_{M} (\nabla^{M}_{X}\nabla^{M}_{Z} (\alpha (s)), Y)\right) .
$$

3) for any vector fields $X, Y\in{\mathcal X}(M)$
and (local) $D$-parallel sections $s,\tilde{s}\in \Gamma (V)$,

\begin{align}
\nonumber\sum_{j=1}^{n}\tilde{\epsilon}_{j}\left( d\alpha (s)^{\flat}\right) (Y, X_{j}) \left(
d\alpha (\tilde{s})^{\flat}\right) (X, X_{j})+\\
\label{reparat} \sum_{i=1}^{r}
\epsilon_{i}\left( d\alpha (s)^{\flat}\right) (Y, \alpha (v_{i})) \left( d\alpha
(\tilde{s})^{\flat}\right) (X, \alpha (v_{i}))=0.
\end{align}

\end{prop}

\begin{proof} From Lemma \ref{levi-civita}, relation (\ref{componente}) and
(\ref{componente2}),
for any (local) $\nabla^{M}$-parallel vector fields $X,Y,Z\in {\mathcal X}(M)$
and $D$-parallel (local) section $s\in\Gamma (V)$,

\begin{align*}
(\nabla_{\tilde{X}}\tilde{Y})^{\mathrm{vert}}&=
\frac{1}{2}\sum_{i=1}^{r}\epsilon_{i}\left( g_{M}(
\nabla^{M,D}_{X}(\alpha ) (v_{i}),Y) +
g_{M}(\nabla^{M,D}_{Y}(\alpha )(v_{i}),X)\right)
v_{i};\\
(\nabla_{\tilde{X}}\tilde{Y})^{\mathrm{hor}}&=-
\frac{1}{2}\sum_{i=1}^{r}\epsilon_{i}\left( g_{M}(
\nabla^{M,D}_{X}(\alpha ) (v_{i}),Y) +
g_{M}(\nabla^{M,D}_{Y}(\alpha )(v_{i}),X)\right)
\alpha (v_{i});\\
(\nabla_{s}\tilde{X})^{\mathrm{vert}}&= \frac{1}{2}\sum_{i=1}^{r}
 \epsilon_{i}(d\alpha (s)^{\flat})(\alpha (v_{i}), X)v_{i};\\
(\nabla_{s}\bar{X})^{\mathrm{hor}}&= \frac{1}{2}
\nabla^{M}_{X}\left(\alpha (s)\right)-\frac{1}{2}\sum_{j=1}^{n}
\tilde{\epsilon}_{j}g_{M}(\nabla^{M}_{X_{j}}(\alpha (s)), X)X_{j} \\
& +\frac{1}{2}\sum_{i=1}^{r}\epsilon_{i}
(d\alpha (s)^{\flat})(X, \alpha (v_{i}))
\alpha (v_{i}).
\end{align*}
Plugging these relations into the expression of $R^{\nabla}$ from
Lemma \ref{curvature} we readily get our claim.
\end{proof}

After this preliminary material, we can now prove our main result from this section (Theorem \ref{corectat} below).
It turns out that if the base $(M, \circ_{M},e_{M}, g_{M})$ is a semisimple Frobenius manifold, the third condition of Proposition
\ref{cond-flatness} simplifies considerably and allows, in the real case, a complete description of all Frobenius structures on $V$, with positive definite
metric, obtained by our method. Recall first that for a semisimple Frobenius manifold the metric is diagonal in canonical coordinates $(u^{1},\cdots , u^{n})$ and may be written
as in (\ref{g-m}) below, in terms of a single function $\eta$, called the metric potential (to simplify notations, we denote by $\eta_{k}$ the derivative
$\frac{\partial\eta}{\partial u^{k}}$; similarly, $\eta_{ij}$ denotes $\frac{\partial^{2}\eta}{\partial u^{i}\partial u^{j}}$, etc).
For more details about semisimple Frobenius manifolds, see e.g. \cite{manin-book}.

\begin{thm}\label{corectat}
Let $(M, \circ_{M}, e_{M}, g_{M})$ be a real semisimple Frobenius
manifold with metric
\begin{equation}\label{g-m}
g_{M} =\sum_{k=1}^{n}\eta_{k}du^{k}\otimes du^{k}
\end{equation}
and non-vanishing rotation coefficients $\gamma_{ij} =
\frac{\eta_{ij}}{\sqrt{\eta_{i}\eta_{j}}}.$ Let $V\rightarrow M$
be a real vector bundle
with a structure of Frobenius algebra
$(\circ_{V}, e_{V}, g_{V})$ along the fibers.
Let
$D$ be a connection on $V$ and $\alpha : V\rightarrow TM$ a
morphism such that
\begin{equation}\label{alpha-alg}
\alpha (e_{V}) =e_{M},\quad \alpha (v_{1}\circ_{V}v_{2}) =\alpha (v_{1})\circ_{M}\alpha
(v_{2}), \quad\forall v_{1},v_{2}\in V.
\end{equation}
Then the almost Frobenius structure $(\circ , e_{V}, g )$ on $V$
defined by this data (see relations (\ref{circ}) and
(\ref{metric})) is Frobenius
with positive definite metric if and only if the following facts hold:\\

1)  the connection $D$ is flat and the Frobenius algebra $(\circ_{V}, e_{V},
g_{V})$ is $D$-parallel;\\

2) the endomorphism $\alpha$ is given by
\begin{equation}\label{key-eqn}
\alpha = \lambda \otimes e_{M}
\end{equation}
where $\lambda \in \Gamma (V^{*})$ is $D$-parallel and
satisfies
\begin{equation}\label{lambda}
\lambda (e_{V})=1,\quad \lambda (s_{1}\circ_{V}s_{2}) = \lambda (s_{1})\lambda
(s_{2}),\quad \forall s_{1}, s_{2}\in \Gamma (V).
\end{equation}

3) $\eta_{k}>0$ for any $1\leq k\leq n$ and $g_{V}-\left( \sum_{k=1}^{n}\eta_{k}\right) \lambda\otimes\lambda$ are positive definite.\\

\end{thm}

\begin{proof}
Recall, from Propositions \ref{multiplication} and
\ref{admissible-cond} and Remark \ref{rem-ss}, that $(V, \circ
,e_{V},g)$ is an $F$-manifold with admissible metric if and only
if $D$ is flat, the Frobenius structure $(\circ_{V}, e_{V},
g_{V})$ along the fibers is $D$-parallel and the morphism $\alpha$
maps any (local) $D$-parallel section $s$ to a vector field which
has constant coefficients in a canonical coordinate system of $(M,
\circ_{M}, e_{M})$. Moreover, remark that $g$ is positive definite
if and only if $g_{M}$ is positive definite (i.e. $\eta_{k}>0$ for
any $1\leq k\leq n$) and $g_{V}-\alpha^{*}g_{M}$ (which is equal
to
$g_{V}-\left(\sum_{k=1}^{n}\eta_{k}\right)\lambda\otimes\lambda$
when $\alpha$ given by (\ref{key-eqn})), is also positive
definite. Therefore, it remains to check that the flatness of $g$
determines the morphism $\alpha$ as in (\ref{key-eqn}). For this,
we use the third condition of Proposition \ref{cond-flatness} with
$\epsilon_{i}=\tilde{\epsilon}_{j}=1$. For any $D$-parallel local
sections $s,\tilde{s}\in \Gamma (V)$,
\begin{equation}\label{alpha-completat}
\alpha (s)= \sum_{k=1}^{n} a_{k}\frac{\partial}{\partial
u^{k}},\quad \alpha(\tilde{s}) = \sum_{k=1}^{n} \tilde{a}_{k}\frac{\partial}{\partial
u^{k}}
\end{equation}
for some constants $a_{k}, \tilde{a}_{k}$. In relation
(\ref{reparat}), let $X:= \frac{\partial }{\partial u^{p}}$ and
$Y:= \frac{\partial}{\partial u^{q}}$. It can be checked that with
these arguments, relation (\ref{reparat}) becomes
\begin{align}
\nonumber&\sum_{j=1}^{n} \frac{\eta_{pj}\eta_{qj}}{\eta_{j}} (a_{j}- a_{q})(\tilde{a}_{j}-\tilde{a}_{p}) +\\
&\label{concl}
\sum_{i=1}^{r}\left( \sum_{s=1}^{n} (a_{s}- a_{q}) \eta_{sq}
du^{s}(\alpha (v_{i}))\right) \left( \sum_{t=1}^{n} (\tilde{a}_{t}- \tilde{a}_{p}) \eta_{tp}
du^{t}(\alpha (v_{i}))\right)
=0 .
\end{align}
In (\ref{concl})   let $p=q$ and $s=\tilde{s}.$ We obtain:
$$
\sum_{j=1}^{n} \frac{\eta_{pj}^{2}}{\eta_{j}} (a_{j}-a_{p})^{2} + \sum_{i=1}^{r} \left(
\sum_{s=1}^{n} (a_{s}-a_{p}) \eta_{ps} du^{s}(\alpha (v_{i}))\right)^{2}=0.
$$
Since $\eta_{pj}$ are non-vanishing and $\eta_{j}>0$ ($g_{M}$ being positive definite), we deduce that
$a_{j}=a_{p}$ for any $1\leq p, j\leq n$, i.e.
$$
\alpha (s) = \lambda (s)e_{M},
$$
where $\lambda (s)$ is a constant.
It follows that  $\lambda \in \Gamma (V^{*})$ is $D$-parallel.
From the algebraic properties (\ref{alpha-alg}) of $\alpha$, we obtain (\ref{lambda}). Conversely,  if $\alpha$ is of the form
(\ref{key-eqn}), with $\lambda\in \Gamma (V^{*})$ $D$-parallel and satisfying (\ref{lambda}),
then $\nabla^{M,D}(\alpha )=0$,
all three conditions from Proposition \ref{cond-flatness} hold and the metric $g$ is flat.
Our claim follows.

\end{proof}

\subsection{A class of Frobenius manifolds}\label{star}

In this section we study a class of Frobenius structures on the
product $M\times\mathbb{K}^{r}$, where $M$ is Frobenius. It is
given by the following theorem.

\begin{thm}\label{trivial}
Let $(M,\circ_{M}, e_{M}, g_{M}, E)$ be a Frobenius manifold with
Euler field such that $L_{E}(g_{M}) =2g_{M}$.
On the vector space $\mathbb{K}^{r}$ consider
any bilinear, commutative, associative multiplication $\circ_{V}$, with unit
$e_{V}$, and $\circ_{V}$-invariant metric $g_{V}$. Let $\lambda
\in (\mathbb{K}^{r})^{*}$ such that
$$
\lambda (v\circ_{V}w) = \lambda (v) \lambda (w), \quad\forall
v,w\in\mathbb{K}^{r},\quad \lambda (e_{V})=1.
$$
Define a multiplication $\circ$ on
$T(M\times\mathbb{K}^{r})$ by
$$
X\circ Y: = X\circ_{M}Y,\quad v\circ w:=v\circ_{V}w \quad X\circ v
:= \lambda (v) X,
$$
where $X,Y\in T_{p}M$, $ v,w\in T_{\vec{\tau}}\mathbb{K}^{r}=
\mathbb{K}^{r}$, and a metric $g$ on $M\times\mathbb{K}^{r}$ by
\begin{equation}\label{g-g}
g(X, Y) = g_{M}(X, Y),\quad g(v,w) = g_{V}(v,w),\quad g(X, v) = \lambda (v) g_{M}(X, e_{M}).
\end{equation}
If the bilinear form
\begin{equation}\label{g-m-v}
g_{M,V}:=g_{V}-g_{M}(e_{M},e_{M}) \lambda\otimes\lambda
\end{equation}
on $\mathbb{C}^{r}$ is non-degenerate, then
$(M\times\mathbb{K}^{r}, \circ , e_{V}, g)$ is a Frobenius
manifold, with Euler field $E + R$ (where
$R_{(p,\vec{\tau})}=\sum_{k=1}^{r}\tau_{k}\frac{\partial}{\partial \tau_{k}}$ is the radial field), and
$L_{E+R}(g) =2g.$
\end{thm}

\begin{proof}
The claim that $(M\times\mathbb{K}^{r}, \circ , e_{V}, g)$ is Frobenius
is straightforward, from
Propositions \ref{multiplication}, \ref{admissible-cond} and
\ref{cond-flatness},  with
$$
\pi :V=M\times\mathbb{K}^{r}\rightarrow M,\quad \pi (p, \tau_{1},
\cdots , \tau_{r}) = p
$$
the trivial bundle, $D$ the trivial standard connection and
$\alpha : V\rightarrow TM$  given by $\alpha (v) =\lambda
(v)e_{M}$, for any $v\in V_{\pi (v)}\cong\mathbb{K}^{r}.$ The
statement involving the Euler fields can be checked directly.
\end{proof}

\begin{rem}{\rm We remark that the Frobenius structure from the
above theorem is more general than those provided by iterations of
adding a variable to a Frobenius manifold (see Proposition
\ref{iterations}). The reason is that the multiplication
$\circ_{V}$ on $\mathbb{K}^{r}$ from the above theorem can be
chosen arbitrarily (as long as it is commutative, associative,
with unit), while in Proposition \ref{iterations} the
multiplication $\circ^{r}$ restricted to vertical vectors
$\frac{\partial}{\partial \tau_{i}}$ has a very precise form (in
particular, it is semisimple, with canonical basis $\{ w_{1}:=
\frac{\partial}{\partial \tau_{1}},
w_{2}:=\frac{\partial}{\partial \tau_{2}}
-\frac{\partial}{\partial\tau_{1}},\cdots ,
w_{r}:=\frac{\partial}{\partial \tau_{r}}
-\frac{\partial}{\partial\tau_{r-1}}\}$).}
\end{rem}

\begin{rem}\label{potential}{\rm
It is easy to relate the potential $F_{M}$ of the Frobenius
structure on $M$ to the potential $F$ of the Frobenius structure
on $M\times \mathbb{K}^{r}$, from Theorem \ref{trivial}. If
$(t^{i})$ are flat coordinates for the metric $g_{M}$, then
$(t^{i}, \tau^{j})$ (with $(\tau^{j})$ the standard coordinate
system of $\mathbb{K}^{r}$) are flat coordinates for $g$ and one
may check that
\begin{align*}
F( t^{i}, {\tau}^{j}) &= F_{M}({t}^{i}) + \frac{1}{2} \left(
\sum_{s=1}^{r}\lambda_{s}\tau_{s}\right) \left(
\sum_{i,j=1}^{n}g_{ij}t^{i}t^{j}\right) + \frac{1}{2}
\left(\sum_{s,k=1}^{r}\lambda_{sk}\tau^{s}\tau^{k}\right) \left(
\sum_{j=1}^{n} \epsilon_{j} t^{j}\right)\\
& + \frac{1}{6} \sum_{s,k,j} g_{skj} \tau^{s}\tau^{k}\tau^{j}
\end{align*}
where
\begin{align*}
&g_{ij}:= g\left(\frac{\partial}{\partial t^{i}},
\frac{\partial}{\partial t^{j}}\right),\quad \lambda_{s}:=
\lambda\left(\frac{\partial}{\partial\tau^{s}}\right),\quad
\lambda_{sk}:=\lambda\left(\frac{\partial}{\partial\tau^{s}}\circ_{V}
\frac{\partial}{\partial\tau^{k}}\right),\\
& g_{skj}=g\left( \frac{\partial}{\partial
\tau^{s}}\circ_{V}\frac{\partial}{\partial \tau^{k}},
\frac{\partial}{\partial \tau^{j}}\right) ,\quad \epsilon_{j} =
g_{M}\left( e_{M}, \frac{\partial}{\partial t^{j}}\right)
\end{align*}
are all constants.}
\end{rem}

Now we show how the Frobenius structure on $M\times\mathbb{K}^{r}$
from Theorem \ref{trivial} can be obtained
from a Saito bundle.
We denote by $\phi^{M}$ and $\phi^{V}$ the Higgs fields determined by
$\circ_{M}$ and $\circ_{V}$, i.e.
$$
\phi^{M}_{X}(Y) = - X\circ_{M}Y,\quad \phi^{V}_{v}(w) = -
v\circ_{V}w,\quad \forall X,Y\in TM,\quad \forall v,w\in
\mathbb{K}^{r},
$$
considered as  $1$-forms on $M$ with values in
$\mathrm{End}(TM\oplus V)$ (acting trivially on $V$ and $TM$,
respectively). Using $\phi^{M}$ and $\phi^{V}$ we define an
$\mathrm{End}(\pi^{*}(TM\oplus V))$-valued $1$-form $\phi^{M,V}$
on $M\times\mathbb{K}^{r}$, which, at $(p,\vec{\tau})\in M\times
\mathbb{K}^{r}$, is given by
\begin{align*}
\phi^{M,V}_{X}(\pi^{*}Y)&= \pi^{*}\phi^{M}_{X}(Y),\quad \phi^{M,V}_{X}\vert_{\pi^{*}V} =0\\
\phi^{M,V}_{v}(\pi^{*}w)&= \pi^{*}\phi^{V}_{v}(w),\quad
\phi^{M,V}_{v}\vert_{\pi^{*}TM} = -\lambda
(v)\mathrm{Id}_{\pi^{*}TM},
\end{align*}
for any $X,Y\in T_{p}M$ and $v,w\in T_{\vec{\tau}}\mathbb{K}^{r}=
\mathbb{K}^{r}=V_{p}.$

\begin{prop}\label{trivial-prop} On
the bundle $\pi^{*}(TM\oplus V)\rightarrow M\times \mathbb{K}^{r}$
consider the following data:\\

1) the Higgs field $\phi^{M,V}$;\\

2) the connection $\pi^{*}(\nabla^{M}\oplus D)$, where
$\nabla^{M}$ is the Levi-Civita connection of $g_{M}$ and $D$ is
the standard trivial connection
on the bundle $V$;\\

3)  the metric $\pi^{*}(g_{M} \oplus g_{M,V})$, where
$g_{M,V}$, defined by (\ref{g-m-v}), is considered as a (constant) metric on the bundle $V$;\\

4) the endomorphism $R^{\prime}_{0}\in\mathrm{End}\left(
\pi^{*}(TM\oplus V)\right)$, which, at a point $(p,\vec{\tau})\in
M\times\mathbb{K}^{r}$, is defined by
\begin{align*}
(R^{\prime}_{0})_{(p,\vec{\tau})} (\pi^{*}X) &:= \pi^{*}(X\circ_{M} E^{M}) + \lambda (\vec{\tau}) \pi^{*}X,\\
(R^{\prime}_{0})_{(p,\vec{\tau})} (\pi^{*}v)&:=
\pi^{*}(v\circ_{V}\vec{\tau}),
\end{align*}
for any $X\in T_{p}M$ and  $v\in
V_{p}=\mathbb{K}^{r}.$\\

5) the endomorphism $R^{\prime}_{\infty}\in\mathrm{End}\left(
\pi^{*}(TM\oplus V)\right)$ defined by
$$
R^{\prime}_{\infty}\vert_{\pi^{*}TM} := \pi^{*}(\nabla^{M}E
-\mathrm{Id}_{TM}),\quad R^{\prime}_{\infty}\vert_{\pi^{*}V}=0.
$$
Then $\pi^{*}(TM\oplus V)$ with this data is a Saito bundle of weight zero.
The Frobenius structure from Theorem \ref{trivial} can
be obtained from this Saito bundle, by
choosing $\pi^{*}(e_{M}+e_{V})$ as a primitive homogeneus section.
\end{prop}

\begin{proof} The proof is straightforward and we omit the details.
We only remark that
$\pi^{*}(e_{M}+e_{V})$ is $\pi^{*}(\nabla^{M}\oplus D)$-parallel,  belongs to the kernel of $R^{\prime}_{\infty}$ (hence is homogeneous),
and
the map
\begin{equation}\label{psi}
\psi : T(M\times\mathbb{K}^{r}) \rightarrow \pi^{*}(TM\oplus V),
\quad Z\rightarrow \psi (Z):= -\phi_{Z}\left( \pi^{*}(e_{M}+ e_{V})\right)
\end{equation}
is given by
\begin{equation}\label{p-s-i}
\psi_{(p,\vec{\tau})} (X) = \pi^{*} (X),\quad \psi_{(p,\vec{\tau})} (v)= \pi^{*}(\lambda (v) e_{M}+
v),
\end{equation}
for any $X\in T_{p}M$ and $v\in
T_{\vec{\tau}}\mathbb{K}^{r}=\mathbb{K}^{r}=V_{p}$, and is an isomorphism.

\end{proof}

\begin{rem}{\rm
In the setting of Proposition \ref{trivial-prop}, assume
that $V=L$ is trivial of rank one.  Changing the trivialization of $V$ if necessary, we may assume that $e_{V}$
(viewed as a vector field on $V$), is
$\frac{\partial}{\partial \tau}$. Then $\lambda = d\tau$ and $\alpha = d\tau\otimes e_{M}.$
Let $g_{V}$ be defined by $g_{V}(e_{V},e_{V})= 1+ g_{M}(e_{M}, e_{M})$.
The resulting Saito structure from Proposition \ref{trivial-prop} coincides with the Saito
structure from Remark \ref{add-var-saito}, the latter with $d=2$, $w=0$, $v=e_{V}$ and $\tilde{g}:= g_{V}- g_{M}(e_{M},e_{M})\lambda\otimes\lambda$.
Therefore, the above proposition is a generalization of adding a variable to a Frobenius manifold, using abstract Saito bundles.}
\end{rem}

\section{$tt^{*}$-geometry on $V$}\label{sect-tt}

The $tt^{*}$-equations may be defined on any holomorphic vector
bundle (with a pseudo-Hermitian metric and holomorphic Higgs
field), but in our considerations  we shall only consider them on
the holomorphic tangent bundle of a complex manifold.

\begin{notation}{\rm
In this section we use slightly different notations from those
employed until now. The reason is that the $tt^{*}$-equations
involve vector fields of type $(1,0)$ and $(0,1)$ as well. In
order to avoid confusion, the holomorphic tangent bundle of a
complex manifold $M$ will be denoted by $T^{1,0}M$ (rather than
$TM$, how it was denoted before). The sheaf of smooth
(respectively holomorphic) vector fields will be denoted by
${\mathcal T}^{1,0}_{M}$ (respectively, ${\mathcal T}_{M}$). For
an holomorphic vector bundle $V\rightarrow M$, $\Gamma (V)$ will
denote, as before, the sheaf of holomorphic sections of $V$.}

\end{notation}

We begin by recalling the $tt^{*}$-equations.

\subsection{The $tt^{*}$-equations}\label{tt-star}

Given a complex manifold $M$ together with an (associative,
commutative, with unit) multiplication $\circ$ on $T^{1,0}M$ and a
pseudo-Hermitian metric $h$, the $tt^{*}$-equations are the
followings:
$$ (\partial^{D}\phi )_{Z_{1}, Z_{2}}=0,\quad R^{D}_{Z_{1},
\bar{Z}_{2}} + [\phi_{Z_{1}}, \phi^{\dagger}_{\bar{Z}_{2}}]=0,
\quad Z_{1},Z_{2}\in {\mathcal T}^{1,0}M
$$
where
$$
(\partial^{D}\phi )_{Z_{1},Z_{2}} := D_{Z_{1}}(\phi_{Z_{2}}) -
D_{Z_{2}}(\phi_{Z_{1}}) - \phi_{[Z_{1},Z_{2}]}.
$$
Above $\phi \in \Omega^{1,0}(M, \mathrm{End}T^{1,0}M)$ is the
Higgs field determined by the multiplication, i.e. $\phi_{X}(Y):=
- X\circ Y$,  $\phi^{\dagger}\in \Omega^{0,1}(M,
\mathrm{End}T^{1,0}M)$ is the $h$-adjoint of $\phi$ and $D$ is the
Chern connection of $h$. In our considerations, the
pseudo-Hermitian metric will be given by $h(Z_{1},Z_{2}) =
g(Z_{1}, kZ_{2})$, where $g$ is a (multiplication)  invariant
metric on $T^{1,0}M$ and  $k: T^{1,0}M\rightarrow T^{1,0}M$ is a
real structure, compatible with $g$, i.e. $g(kZ_{1},kZ_{2}) =
\overline{g(Z_{1},Z_{2})}$, for any $Z_{1},Z_{2}\in T^{1,0}M$.
(Recall that a real structure on a vector bundle is a
fiber-preserving complex anti-linear involution). In this case,
$\phi^{\dagger}_{\bar{Z}}= k \phi_{Z}k$, for any $Z\in T^{1,0}M.$
The compatibility condition insures that $h$ defined as above in
terms of $g$ and $k$ is indeed a pseudo-Hermitian metric.

\begin{ex}\label{capi}{\rm On any semisimple complex Frobenius manifold
$(M, \circ , e, g)$ one can define a diagonal real structure, such
that the $tt^{*}$-equations hold (see Theorem 2.4 of \cite{lin}).
More precisely, let $\eta$ be the metric potential in canonical
coordinates $(u^{1}, \cdots , u^{n})$, i.e. $g\left(
\frac{\partial}{\partial u^{i}}, \frac{\partial}{\partial
u^{j}}\right) = \delta_{ij} \eta_{i}$, for any $i,j$ (with
$\eta_{i}= \frac{\partial\eta}{\partial u^{i}}$). Define a real
structure on $T^{1,0}M$, compatible with $g$,  by
$k\left(\frac{\partial}{\partial u^{i}} \right) = \frac{
|\eta_{i}|}{\eta_{i}}\frac{\partial}{\partial u^{i}}$, for any
$i$.  The Chern connection $D$ of the associated pseudo-Hermitian
metric $h= g(\cdot , k\cdot )$ is given by
\begin{equation}\label{chern-diag}
D_{Z}\left(\frac{\partial}{\partial u^{i}}\right) =
Z\left(\mathrm{log}| \eta_{i}| \right) \frac{\partial}{\partial
u^{i}},\quad Z\in {\mathcal T}^{1,0}_{M},\quad 1\leq i\leq n
\end{equation}
and one may show that the $tt^{*}$-equations hold. Remark that in
this example the second $tt^{*}$-equation decouples: $D$ is flat
and the Higgs field commutes with its $h$-adjoint. Other examples
of $tt^{*}$-structures may be constructed on the base space of the
universal unfolding of various Laurent polynomials
\cite{sabbah-art}. There is also a real version of the
$tt^{*}$-equations, which relate harmonic Higgs bundles with
special geometries, see e.g. \cite{cortes,sch}.}
\end{ex}

\subsection{The main result}

Let $\pi : V \rightarrow M$ be an holomorphic vector bundle whose
base has an holomorphic Frobenius structure $(\circ_{M},e_{M},
g_{M})$, the typical fiber has a Frobenius algebra
$(\circ_{V},e_{V}, g_{V})$, and which comes equipped with an
holomorphic connection $D$ and an holomorphic morphism $\alpha :
V\rightarrow T^{1,0}M$, preserving the multiplications and the
unit fields. Let $\circ$ and $g$ be the induced multiplication and
metric on the manifold $V$, defined as usual by (\ref{circ}) and
(\ref{metric}). We denote by $\phi$, $\phi^{M}$ and $\phi^{V}$ the
Higgs fields associated to $\circ$, $\circ_{M}$ and $\circ_{V}$
respectively.

\begin{assumption}{\rm In the following considerations, we  assume that $(V,\circ ,e_{V})$ is an $F$-manifold, i.e. all
conditions from Proposition \ref{multiplication} are satisfied -
in particular, $D$ is flat.  (This assumption is not restrictive:
our goal is to construct solutions to the $tt^{*}$-equations on
$(V, \circ , e_{V},g)$; but, as shown in Lemma 4.3 of
\cite{referee2}, the first $tt^{*}$-equation implies the
$F$-manifold condition).}
\end{assumption}

Let $k_{M}$ be a real structure on $T^{1,0}M$, compatible with
$g_{M}$ and $h_{M}:= g_{M}(\cdot , k_{M}\cdot )$ the associated
pseudo-Hermitian metric. Similarly, let $k_{V}$ be a real
structure on the bundle $V$, compatible with $g_{V}$, and
$h_{V}:= g_{V}(\cdot , \cdot )$. Using $k_{M}$ and $k_{V}$ we
construct a real structure $k$ on $T^{1,0}V$, which on the
$D$-horizontal subbundle of $T^{1,0}V$ coincides with $k_{M}$ and
on the vertical subbundle coincides with $k_{V}.$ We assume that
\begin{equation}\label{compatibilitate}
\alpha \circ k_{V} = k_{M}\circ\alpha
\end{equation}
which means that $k$ is compatible with $g$ (easy check). Let $h:=
g(\cdot , k\cdot )$, given by
\begin{align*}
h(\tilde{X}, \tilde{Y}) = h_{M}(X,Y),\quad h(\tilde{X}, v)= h_{M}
(X, \alpha (v))\\
h(v_{1},v_{2}) = h_{V}(v_{1},v_{2}), \quad h(v,\tilde{X}) =
h_{M}(\alpha (v), X),
\end{align*}
for any $X,Y\in T_{p}^{1,0}M$ and $v,v_{1},v_{2}\in V_{p}$ ($p\in
M$). Our main result is the following.

\begin{thm}\label{main-thm-tt} In the above setting, the $tt^{*}$-equations
hold on $(V, \circ ,h)$ if and only if they hold on $(M,
\circ_{M}, h_{M})$ and the following additional conditions
are true:\\

i) for any local $D$-parallel section $s\in \Gamma (V)$,
\begin{equation}\label{tt-1}
D^{M}_{X} \left( \alpha (s)\circ_{M} Z\right) = \alpha
(s)\circ_{M} D^{M}_{X} (Z),\quad \forall X, Z\in {\mathcal T}_{M},
\end{equation}
where $D^{M}$ is the Chern connection of $h_{M}$;\\

ii) The $(1,0)$-part of the Chern connection $D^{V}$ of $h_{V}$ is
related to $D$ by
\begin{equation}\label{tt-2}
D^{V}_{X}s = D_{X}s + (D^{V}_{X}e_{V}) \circ_{V} s,\quad \forall
X\in {\mathcal T}_{M},\quad \forall s\in \Gamma (V);
\end{equation}

iii) for any $X\in {\mathcal T}_{M}$ and $s,s_{1}, s_{2}\in \Gamma
(V)$,
\begin{equation}\label{crosete-tt}
[ \phi^{M}_{X}, k_{M}\phi^{M}_{\alpha (s)} k_{M}]=0,\quad [
\phi^{V}_{s_{1}}, k_{V}\phi^{V}_{s_{2}}k_{V}]=0;
\end{equation}

iv) For any $X, Y\in {\mathcal T}_{M}$ and $s,s_{1}\in \Gamma
(V)$,
\begin{align}
\nonumber&h_{V} \left( R^{D^{V}}_{X,\bar{Y}}s, s_{1}\right) -
h_{M} \left( R^{D^{M}}_{X,\bar{Y}}\alpha (s), \alpha
(s_{1})\right)=\\
\label{curvature}&  h_{M}\left( {\mathcal D}_{X}(\alpha )(s) ,
D^{M}_{Y}(\alpha (s_{1})) - \alpha (D^{M,V}_{Y}s_{1})\right) ,
\end{align}
where $D^{M,V}$ is the Chern connection of $h_{V}-\alpha^{*}h_{M}$
and
$$
{\mathcal D}_{X}(\alpha )(s):= D^{M}_{X}(\alpha (s)) - \alpha
(D^{V}_{X}s).
$$

\end{thm}

We divide the proof of the above theorem in two stages. In a first
stage, we find the conditions for the first $tt^{*}$-equation to
hold on $(V,\circ ,h)$ (see Proposition \ref{prop-1-tt}). In  a
second stage, we compute the curvature of the Chern connection of
$h$ (see Lemma \ref{lem-tt}). This will readily imply the
conditions for the second $tt^{*}$-equation to hold on $(V,\circ
,h)$ and will conclude the proof of Theorem \ref{main-thm-tt}.
Details are as follows.

\begin{prop}\label{prop-1-tt} The first $tt^{*}$-equation holds on $(V,\circ , h)$
if and only if it holds on $(M, \circ_{M}, h_{M})$ and conditions
(\ref{tt-1}) and (\ref{tt-2}) are true.
\end{prop}

\begin{proof} Let $D^{c}$ be the Chern connection of $h$.
Using the relation
$$
{\mathcal W}_{1} h \left({\mathcal W}_{2}, {\mathcal W}_{3}\right)
= h\left( D^{c}_{\mathcal W_{1}}(\mathcal W_{2}), {\mathcal
W}_{3}\right) ,\quad \forall \mathcal W_{i}\in {\mathcal T}_{V}
$$
we obtain the expression of $D^{c}$ as follows:
\begin{align}
\nonumber& D^{c}_{\tilde{X}}\tilde{Y} = \left[
D^{M}_{X}Y\right]^{\widetilde{}},\quad h\left( D^{c}_{\tilde{X}}s,
\tilde{Y}\right) = h_{M}\left( D^{M}_{X}(\alpha (s)), Y\right)\\
\label{chern}& D^{c}_{s} \tilde{Y} = D^{c}_{s_{1}}s_{2} =0,\quad
h\left( D^{c}_{\tilde{X}}s, {s}_{1}\right) = h_{V}\left(
D^{V}_{X}s, s_{1}\right),
\end{align}
for any $s,s_{1},s_{2}\in \Gamma (V)$ and $X, Y\in {\mathcal
T}_{M}$. Using these relations, together with
$$
(\partial^{D^{c}} \phi )_{\mathcal W_{1}, {\mathcal W}_{2}} =
D^{c}_{\mathcal W_{1}} (\phi_{\mathcal W_{2}}) - D^{c}_{\mathcal
W_{2}} (\phi_{\mathcal W_{1}} ) - \phi_{[\mathcal W_{1},\mathcal
W_{2}]}
$$
and the flatness of $D$, we obtain
\begin{align*}
(\partial^{D^{c}}\phi )_{\tilde{X},\tilde{Y}}(\tilde{Z})&
=\left[ (\partial^{D^{M}} \phi^{M})_{X,Y}(Z)\right]^{\widetilde{}}\\
(\partial^{D^{c}}\phi )_{\tilde{X},\tilde{Y}}(s) &=  \left[ (
\partial^{D^{M}}\phi^{M}_{\alpha (s)})_{X,Y}\right]^{\widetilde{}} -
\tilde{X}\circ D^{c}_{\tilde{Y}}s + \tilde{Y}\circ
 D^{c}_{\tilde{X}}s\\
(\partial^{D^{c}}\phi )_{s_{1},s_{2}}(\tilde{Z})&=
(\partial^{D^{c}}\phi )_{s_{1},s_{2}}(s)=0
\end{align*}
for any $X,Y,Z\in {\mathcal T}_{M}$, $s,s_{1}\in \Gamma (V)$,
where
$$
\left( \partial^{D^{M}}\phi^{M}_{\alpha (s)}\right)_{X,Y} := -
D^{M}_{X}\left( \alpha (s)\circ_{M} Y\right) + D^{M}_{Y}\left(
\alpha (s)\circ_{M} X\right) + \alpha (s)\circ_{M} [ X,Y].
$$
If, moreover, $s_{1}$ is $D$-parallel, then
\begin{align*}
(\partial^{D^{c}}\phi )_{s_{1},\tilde{X}}(\tilde{Z}) &= \left[
D^{M}_{X}\left(\alpha (s_{1})\circ_{M}Z\right)  - \alpha (s_{1})
\circ_{M}
D^{M}_{X} (Z)\right]^{\widetilde{}}\\
(\partial^{D^{c}}\phi )_{s_{1},\tilde{X}}(s)& = D^{c}_{\tilde{X}}
(s\circ_{V}s_{1}) - s_{1}\circ D^{c}_{\tilde{X}}s.\\
\end{align*}
It follows that $\partial^{D^{c}}\phi =0$ if and only if
$\partial^{D^{M}}\phi^{M}=0$, relation (\ref{tt-1}) holds and the
following two relations hold as well:
\begin{align}
\nonumber& D^{c}_{\tilde{X}} (s\circ_{V}s_{1}) -
s_{1}\circ D^{c}_{\tilde{X}}s=0\\
\label{delucrat}& \left[ (\partial^{D^{M}}\phi^{M}_{\alpha
(s)})_{X,Y}\right]^{\widetilde{}} - \tilde{X}\circ
D^{c}_{\tilde{Y}}s + \tilde{Y}\circ
 D^{c}_{\tilde{X}}s=0,
\end{align}
for any $X,Y\in {\mathcal T}_{M}$, $s,s_{1}\in \Gamma (V)$ and
$D(s_{1})=0.$ In the remaining part of the proof we assume that
$\partial^{D^{M}}\phi^{M}=0$ and that relation (\ref{tt-1}) is
true and we show that relations (\ref{delucrat}) are equivalent to
(\ref{tt-2}). We begin with the first relation (\ref{delucrat}).
Notice that: for any $s_{2}\in \Gamma (V)$,
\begin{align} \nonumber h\left( s_{1}\circ D^{c}_{\tilde{X}}
s, s_{2}\right) &= g\left( D^{c}_{\tilde{X}}s, s_{1}\circ k_{V}
(s_{2})\right) = h\left( D^{c}_{\tilde{X}} s, k_{V}
(s_{1}\circ_{V}
k_{V}(s_{2}))\right)\\
\nonumber &= h_{V}\left( D^{V}_{X}s, k_{V}\left( s_{1}\circ_{V}
k_{V}(s_{2})\right)\right)= g_{V}\left( D^{V}_{X}s, s_{1}\circ_{V}k_{V}(s_{2})\right)\\
\label{relo}&= h_{V} \left( s_{1}\circ_{V} D^{V}_{X}s,
s_{2}\right) ,
\end{align}
where we used (\ref{chern}). Also from (\ref{chern}),
\begin{equation}\label{notices-1} h\left( D^{c}_{\tilde{X}}
(s\circ s_{1}), s_{2}\right)= h_{V}\left( D^{V}_{X} \left(
s\circ_{V}s_{1}\right), s_{2}\right) .
\end{equation}
Combining (\ref{relo}) and (\ref{notices-1}) we obtain
\begin{equation}\label{ro-1}
h\left( D^{c}_{\tilde{X}} (s\circ_{V}s_{1}) - s_{1}\circ
D^{c}_{\tilde{X}}s,s_{2} \right) = h_{V}\left(
D^{V}_{X}(s\circ_{V} s_{1}) - s_{1}\circ_{V} D^{V}_{X}s,
s_{2}\right).
\end{equation}
A similar computation shows that
$$
h\left( D^{c}_{\tilde{X}} (s\circ_{V}s_{1}) - s_{1}\circ
D^{c}_{\tilde{X}}s, \tilde{Z}\right)  = h_{M}\left(
D^{M}_{X}\left( \alpha (s\circ_{V}s_{1})\right) - \alpha
(s_{1})\circ_{M} D^{M}_{X} (\alpha (s)), Z\right) ,
$$
which follows from (\ref{tt-1}). We proved that the first relation
(\ref{delucrat}) gives
\begin{equation}\label{d-s-1}
D^{V}_{X} (s\circ_{V} s_{1}) = (D^{V}_{X}s)\circ_{V} s_{1},
\end{equation}
for any $X\in {\mathcal T}_{M}$, $s,s_{1}\in \Gamma (V)$ with
$D(s_{1})=0$, and this  is equivalent to (\ref{tt-2}) (easy
check). It remains to consider the second relation
(\ref{delucrat}). For this, one shows that
\begin{align*}
& h\left( \tilde{X}\circ D^{c}_{\tilde{Y}} s, s_{2}\right) =
h_{M}\left( X\circ_{M}
D^{M}_{Y}\left(\alpha (s)\right),\alpha (s_{2})\right)\\
& h\left( \tilde{X}\circ D^{c}_{\tilde{Y}}s, \tilde{Z}\right) =
h_{M}\left( X\circ_{M} D^{M}_{Y}\left(\alpha (s)\right), Z\right)
\end{align*}
which imply
\begin{align*}
&h\left(  \left[ (\partial^{D^{M}}\phi^{M}_{\alpha
(s)})_{X,Y}\right]^{\widetilde{}} - \tilde{X}\circ
D^{c}_{\tilde{Y}}s + \tilde{Y}\circ
 D^{c}_{\tilde{X}}s ,s_{2}\right) \\
 &=h_{M}\left( (\partial^{D^{M}}\phi^{M})_{X,Y}(\alpha (s)),
 \alpha (s_{2})\right)\\
&h\left( \left[ (\partial^{D^{M}}\phi^{M}_{\alpha
(s)})_{X,Y}\right]^{\widetilde{}} - \tilde{X}\circ
D^{c}_{\tilde{Y}}s  + \tilde{Y}\circ
 D^{c}_{\tilde{X}}s ,Z\right)\\
&=h_{M}\left( (\partial^{D^{M}}\phi^{M})_{X,Y}(\alpha (s)),
 Z\right) ,
\end{align*}
for any $X,Y\in {\mathcal T}^{1,0}_{M}$ and $s,s_{1},s_{2}\in
\Gamma (V)$, with $D(s_{1})=0.$ Thus,  the second relation
(\ref{delucrat}) is a consequence of the first $tt^{*}$-equation
$\partial^{D^{M}}\phi^{M}=0$. Our claim follows.

\end{proof}

\begin{lem}\label{lem-tt} The curvature of the Chern connection $D^{c}$ of $h$
has the following expression: for any $X,Y,Z\in {\mathcal T}_{M}$
and $s,s_{1},s_{2}\in \Gamma (V)$,
\begin{align*}
&R^{D^{c}}_{\tilde{X},\overline{\tilde{Y}}} \tilde{Z} =\left[
R^{D^{M}}_{X,\bar{Y}}Z\right]^{\widetilde{}},\quad h\left(
R^{D^{c}}_{\tilde{X}, \overline{\tilde{Y}}}s, \tilde{Z}\right) =
h_{M}\left( R^{D^{M}}_{X, \bar{Y}}\alpha
(s), Z\right) ,\\
&R^{D^{c}}_{s, \bar{\tilde{X}}}\tilde{Y}=
R^{D^{c}}_{\tilde{X},\bar{s}}\tilde{Y}= R^{D^{c}}_{s_{1},
\bar{s}_{2}}\tilde{Y} = R^{D^{c}}_{s_{1},\bar{\tilde{X}}} s =
R^{D^{c}}_{\tilde{X}, \bar{s}_{1}} s =
R^{D^{c}}_{s_{1},\bar{s}_{2}}s=0
\end{align*}
and
\begin{equation}\label{r-d}
h\left( R^{D^{c}}_{\tilde{X}, \overline{\tilde{Y}}}s, s_{1}\right)
= h_{V} \left( R^{D^{V}}_{X,\bar{Y}}s, s_{1}\right) + h_{M} \left(
{\mathcal D}_{X}(\alpha )(s),D^{M}_{Y} (\alpha (s_{1}))-\alpha
(D^{M,V}_{Y}s_{1})\right) .
\end{equation}

\end{lem}

\begin{proof}
We only show (\ref{r-d}), the other components of $R^{D^{c}}$ can
be obtained by straightforward computations, using (\ref{chern}).
To prove (\ref{r-d}), one first notices, from
$$
h\left( D^{c}_{\tilde{X}}s, \tilde{Z}\right) = h_{M}\left(
D^{M}_{X}(\alpha (s)), Z\right), \quad h\left( D^{c}_{\tilde{X}}s,
{s}_{1}\right) = h_{V}\left( D^{V}_{X}s, s_{1}\right),
$$
that
$$
D^{c}_{\tilde{X}} s = s_{0} +\tilde{X}_{0}
$$
where $s_{0}$ is a section of $V$ (not necessarily holomorphic),
determined by
\begin{equation}\label{ver}
\left( h_{V}-\alpha^{*}h_{M}\right) (s_{0}, s_{1}) = h_{V}\left(
D^{V}_{X}s, s_{1}\right) - h_{M}\left( D^{M}_{X}(\alpha (s)),
\alpha (s_{1})\right),
\end{equation}
for any $s_{1}\in \Gamma (V)$, and $X_{0}\in {\mathcal
T}^{1,0}_{M}$ is determined by
\begin{equation}\label{hor}
h_{M}(X_{0}, Z) = h_{M}(D^{M}_{X}(\alpha (s)), Z) - h_{M}( \alpha
(s_{0}), Z),
\end{equation}
for any $Z\in {\mathcal T}_{M}$.  Now,
\begin{equation}\label{curv-dc}
h\left( R^{D^{c}}_{\tilde{X},\bar{\tilde{Y}}}s, s_{1}\right) = -
h\left( \bar{\partial}_{\bar{\tilde{Y}}}
D^{c}_{\tilde{X}}s,s_{1}\right) = - h_{V}\left(
\bar{\partial}_{\bar{Y}}s_{0} ,s_{1}\right) -
h_{M}\left(\bar{\partial}_{\bar{Y}} X_{0}, \alpha (s_{1})\right),
\end{equation}
where in the first equality we used that $X, Y\in {\mathcal
T}_{M}$ and the flatness of $D$ (hence
$[\tilde{X},\bar{\tilde{Y}}]=0$), and also that $s$ is holomorphic
(hence $\bar{\partial}_{\bar{Y}}s=0$). Therefore, we need to
compute $\bar{\partial}_{\bar{Y}}s_{0}$ and
$\bar{\partial}_{\bar{Y}} X_{0}.$ By applying
$\bar{\partial}_{\bar{Y}}$ to (\ref{ver}), it is easy to see that
\begin{align}
\nonumber\left( h_{V}-\alpha^{*}h_{M}\right) \left(
\bar{\partial}_{\bar{Y}}s_{0}, s_{1}\right)&= h_{M}\left(
R^{D^{M}}_{X,\bar{Y}}\alpha (s), \alpha (s_{1})\right)
 - h_{V}\left( R^{D^{V}}_{X,\bar{Y}}s, s_{1}\right)\\
\nonumber & + h_{M}\left( D^{M}_{X}(\alpha (s)), \alpha
 (D^{M,V}_{Y}s_{1}) - D^{M}_{Y}(\alpha (s_{1}))\right)\\
\label{aj}& + h_{V}\left( D^{V}_{X}s, D^{V}_{Y}s_{1} -
D^{M,V}_{Y}s_{1}\right).
\end{align}
Similarly, one computes,
\begin{equation}
h_{M}\left( \bar{\partial}_{\bar{Y}}X_{0}, Z\right) = -
h_{M}\left( R^{D^{M}}_{X,\bar{Y}}\alpha (s), Z\right) -
h_{M}\left( \alpha (\bar{\partial}_{\bar{Y}}s_{0}), Z\right) .
\end{equation}
Combining this relation with (\ref{curv-dc}), we obtain
\begin{align*}
& h\left( R^{D^{c}}_{\tilde{X},\bar{\tilde{Y}}}s, s_{1}\right) = -
h_{V}\left( \bar{\partial}_{\bar{Y}}s_{0}, s_{1}\right) -
h_{M}\left( \bar{\partial}_{\bar{Y}}X_{0}, \alpha
(s_{1})\right)\\
& = (- h_{V}+\alpha^{*}h_{M})
\left(\bar{\partial}_{\bar{Y}}s_{0},s_{1}\right) + h_{M}\left(
R^{D^{M}}_{X,\bar{Y}}\alpha (s),\alpha (s_{1})\right) ,
\end{align*}
or, using (\ref{aj}),
\begin{align}
\nonumber h\left( R^{D^{c}}_{\tilde{X},\bar{\tilde{Y}}}s,
s_{1}\right) &= - h_{V}\left( D^{V}_{X}s, D^{V}_{Y} s_{1}\right) +
h_{M}\left
(D^{M}_{X}(\alpha (s)), D^{M}_{Y}(\alpha (s_{1}))\right)\\
\nonumber &+ h_{V} \left( D^{V}_{X}s, D^{M,V}_{Y}s_{1}\right) -
h_{M}\left( D^{M}_{X}( \alpha (s)) , \alpha
(D^{M,V}_{Y}s_{1})\right)\\
\label{prel-r}&+  h_{V}\left( R^{D^{V}}_{X,\bar{Y}}s,s_{1}\right)
.
\end{align}
Denote by $E_{1}$ and $E_{2}$  the term on the first line,
respectively the second line, on the right hand side of
(\ref{prel-r}). Remark that
\begin{align*}
E_{1}= & -\bar{Y} h_{V}\left( D^{V}_{X}s, s_{1}\right) + h_{V}
\left( \bar{\partial}_{\bar{Y}} D^{V}_{X}s, s_{1}\right) +
\bar{Y}h_{M}\left( D^{M}_{X}(\alpha (s)), \alpha (s_{1})\right)\\
& - h_{M}\left( \bar{\partial}_{\bar{Y}}
D^{M}_{X}(\alpha (s)), \alpha (s_{1})\right)\\
& = - \bar{Y} h_{V}\left( D^{V}_{X}s, s_{1}\right) - h_{V}\left(
R^{D^{V}}_{X,\bar{Y}}(s), s_{1}\right) + \bar{Y}h_{M}\left(
D^{M}_{X}
(\alpha (s)), \alpha (s_{1})\right)\\
& + h_{M} \left( R^{D^{M}}_{X,\bar{Y}}\alpha (s) , \alpha
(s_{1})\right) .
\end{align*}
A similar computation shows that
\begin{align*}
E_{2} &= \left( h_{V}-\alpha^{*}h_{M}\right) \left( D^{V}_{X}s,
D^{M,V}_{Y}s_{1}\right) - h_{M} \left({\mathcal D}_{X}(\alpha
)(s),\alpha (D^{M,V}_{Y}s_{1})\right)\\
&= \bar{Y}h_{V}\left( D^{V}_{X}s,s_{1}\right) - \bar{Y}
h_{M}\left( \alpha (D^{V}_{X}s),\alpha (s_{1})\right) +
(h_{V}-\alpha^{*}h_{M})\left( R^{D^{V}}_{X,\bar{Y}}s,
s_{1}\right)\\
& - h_{M}\left( {\mathcal D}_{X}(\alpha )(s) , \alpha (
D^{M,V}_{Y}s_{1})\right)
\end{align*}
and we obtain
\begin{equation}\label{math-e}
E_{1}+ E_{2} = h_{M}\left( {\mathcal D}_{X}(\alpha ) (s),
D^{M}_{Y}(\alpha (s_{1})) - \alpha (D^{M,V}_{Y}s_{1})\right) .
\end{equation}
Combining (\ref{prel-r}) with (\ref{math-e}) we obtain
(\ref{r-d}), as required.
\end{proof}

Theorem \ref{main-thm-tt} follows from Proposition
\ref{prop-1-tt}, Lemma \ref{lem-tt} and the following brackets:
\begin{align*}
&[\phi_{\tilde{X}}, k \phi_{\tilde{Y}}k] (\tilde{Z})
=[\phi^{M}_{X}, k_{M}\phi^{M}_{Y}k_{M}] (Z)^{\widetilde{}},\
[\phi_{\tilde{X}}, k \phi_{\tilde{Y}}k] (s)= [ \phi^{M}_{X},
k_{M}\phi^{M}_{Y}k_{M}](\alpha (s))^{\widetilde{}}\\
&[\phi_{\tilde{X}}, k \phi_{s}k] (\tilde{Z}) =[\phi^{M}_{X},
k_{M}\phi^{M}_{\alpha (s)} k_{M}] (Z)^{\widetilde{}},\
[\phi_{\tilde{X}}, k \phi_{s}k] ({s}_{1}) =[\phi^{M}_{X},
k_{M}\phi^{M}_{\alpha (s)} k_{M}] (\alpha (s_{1}))^{\widetilde{}}\\
&[\phi_{s}, k\phi_{\tilde{X}}k] (\tilde{Z}) = [ \phi^{M}_{\alpha
(s)}, k_{M}\phi^{M}_{X}k_{M}] (Z)^{\widetilde{}}, \ [\phi _{s},
k\phi_{\tilde{X}}k] (s_{1})= [ \phi^{M}_{\alpha
(s)},k_{M}\phi^{M}_{X}k_{M}]\left( \alpha
(s_{1})\right)^{\widetilde{}}\\
&[\phi_{s_{1}}, k \phi_{s_{2}} k] (\tilde{Z})= [\phi^{M}_{\alpha
(s_{1})}, k_{M} \phi^{M}_{\alpha (s_{2})}k_{M}](Z)^{
\widetilde{}},\ [\phi_{s_{1}}, k \phi_{s_{2}} k] (s_{3}) = [
\phi^{V}_{s_{1}}, k_{V} \phi^{V}_{s_{2}}k_{V}](s_{3})
\end{align*}

\begin{rem}\label{cov-chern}{\rm In the setting of Theorem \ref{main-thm-tt}, it
may be checked that the Chern connection $D^{c}$ of $h$ preserves
$g$, i.e. $D^{c}(g)=0$, if and only if $D^{M}(g_{M})=0$ and
$D^{V}(g_{V})=0.$} \end{rem}

\begin{ex}\label{detailed-ex} {\rm
We now construct an example where all conditions from Proposition
\ref{multiplication} and Theorem \ref{main-thm-tt} are satisfied
and hence the $tt^{*}$-equations hold on $V$. Consider a complex
semisimple Frobenius manifold $(M, \circ_{M} ,e_{M},g_{M})$ with
metric potential $\eta$ in canonical coordinates $(u^{1}, \cdots ,
u^{n}).$ Define a diagonal real structure $k_{M}$ on $T^{1,0}M$,
like in Example \ref{capi}. Let $V\rightarrow M$ be a rank
$n$-holomorphic vector bundle and assume there is an (holomorphic)
bundle isomorphism $\alpha : V \rightarrow T^{1,0}M$. Identifying
$V$ with $T^{1,0}M$ using $\alpha$, we obtain a multiplication
$\circ_{V}$ and a real structure $k_{V}$ on $V$, induced by
$\circ_{M}$ and $k_{M}$ respectively. Let $g_{V}:= k_{0}
\alpha^{*}g_{M}$, where $k_{0}\in \mathbb{R}\setminus \{ 1\}$ is
fixed. Note  that $g_{V}-\alpha^{*}g_{M}= (k_{0}-1)
\alpha^{*}g_{M}$ is non-degenerate and  $g_{V}$ is compatible with
$k_{V}.$ It remains to define the connection $D$. It is determined
by the condition that an (holomorphic) section $s$ of $V$ is
$D$-parallel if and only if $\alpha (s)$ has constant coefficients
in the coordinate system $(u^{1}, \cdots , u^{n})$. In this
setting, we claim that all conditions from Proposition
\ref{multiplication} and Theorem \ref{main-thm-tt} are satisfied.
This may be checked easily. For example, relation (\ref{tt-1})
follows from
$$
D^{M}_{X}\left( \frac{\partial}{\partial u^{i}}\circ_{M}Z \right)
= \frac{\partial}{\partial u^{i}} \circ_{M} D^{M}_{X}(Z), \quad
\forall X,Z\in {\mathcal T}_{M}
$$
which, in turn, is a consequence of (\ref{chern-diag}). Similarly,
to prove (\ref{curvature}) one remarks that the pseudo-Hermitian
metric $h_{V}:= g_{V}(\cdot , k_{V}\cdot )$ is given by $h_{V}=
k_{0} \alpha^{*}h_{M}$ and its Chern connection $D^{V}$ is related
to $D^{M}$ by
\begin{equation}\label{chern-i}
D_{X}^{V}s = \alpha^{-1} D^{M}_{X}\left(\alpha (s)\right) ,\quad
\forall s\in \Gamma (V), \quad \forall X\in {\mathcal T}_{M}
\end{equation}
and  is flat (because $D^{M}$ is flat, see Example \ref{capi}).
From (\ref{chern-i}),  ${\mathcal D}(\alpha )=0$ and relation
(\ref{curvature}) follows. In a similar way one may check the
other conditions from Proposition \ref{multiplication} and Theorem
\ref{main-thm-tt} and hence the $tt^{*}$-equations hold on $V$.
Finally, remark that $D(g_{V})\neq 0$ (unless $\eta_{i}$ are
constant) and from Proposition \ref{admissible-cond}, $g$ is not
an admissible metric on the $F$-manifold $(V,\circ ,e_{V}).$ From
Remark \ref{cov-chern}, the Chern connection of $h$ preserves $g$
(because  $D^{M}(g_{M})=0$ and $D^{V}(g_{V})=0$.}
 \end{ex}

It is natural to ask if the Frobenius manifold
$(V=M\times\mathbb{C}^{r},\circ ,e_{V},g)$ from Theorem
\ref{trivial} (with $\mathbb{K}=\mathbb{C}$) may be given a real
structure $k$ in the framework of this section, such that the
$tt^{*}$-equations hold. It turns out that imposing the
$tt^{*}$-equations in this setting is a very strong condition, due
to the very special form of the morphism $\alpha$ from (the proof
of) Theorem \ref{trivial}. More precisely, remark that if $k_{M}$
is a real structure on $M$ (compatible with $g_{M}$) and $k_{V}$
is a real structure on the trivial bundle $V=
M\times\mathbb{C}^{r}\rightarrow M$ (compatible with $g_{V}$),
such that $\alpha\circ k_{V} = k_{M}\circ\alpha$, with
$\alpha=\lambda\otimes e_{M}$ as in the proof of Theorem
\ref{trivial}, then $k_{M}(e_{M})$ is a constant multiple of
$e_{M}$. However, this fact, together with the $tt^{*}$-equations
on $M$ (which, according to Theorem \ref{main-thm-tt}, is a
necessary condition for the $tt^{*}$-equations to hold on $V$)
impose strong restrictions on the base Frobenius manifold
$(M,\circ_{M},e_{M},g_{M})$, as shown in the following proposition
(the condition $D(g)=0$ below  holds for a large class of
$tt^{*}$-structures, e.g. CV-structures or harmonic Frobenius
structures \cite{referee2, sabbah-art}).

\begin{prop}\label{second} Assume that the $tt^{*}$-equations hold on
a complex Frobenius manifold $(M, \circ , e, g)$ with real
structure $k$, compatible with $g$,  and  $k(e) =\mu e$, where
$\mu$ is a constant. Assume, moreover, that the Chern connection
$D$ of $h = g(\cdot , k\cdot )$ preserves $g$. Then the Frobenius
manifold $(M,\circ , e, g)$ is trivial, $k$ is constant in flat
coordinates for $g$ and satisfies
\begin{equation}\label{morphism}
k(Z_{1}\circ Z_{2}) = \bar{\mu} k(Z_{1})\circ k(Z_{2}),\quad
\forall Z_{1}, Z_{2}\in {\mathcal T}^{1,0}_{M}.
\end{equation}
\end{prop}

\begin{proof} Since $D$ preserves $g$ and $h$, it preserves $k$ as
well (i.e. $D_{X}(k(Y)) = k (D_{\bar{X}}Y)$, for any $Y\in
{\mathcal T}^{1,0}_{M}$ and complex vector field $X$). From
$k(e)=\mu e$,
$$
\mu D_{Z}e = D_{Z}\left( k(e)\right) = k \left(
\bar{\partial}_{\bar{Z}}e\right) =0,\quad \forall Z\in {\mathcal
T}^{1,0}_{M}
$$
where we used that $e\in {\mathcal T}_{M}.$ Thus, $D(e)=0$,
$R^{D}_{Z_{1}, \bar{Z}_{2}} (e)=0$ and, from the second
$tt^{*}$-equation,
\begin{equation}\label{ad-j}
[\phi_{Z_{1}} , \phi^{\dagger}_{\bar{Z}_{2}}](e)=0,\quad \forall
Z_{1}, Z_{2}\in {\mathcal T}^{1,0}_{M},
\end{equation}
where $\phi_{X}(Y) = - X\circ Y$ is the Higgs field and
$\phi^{\dagger}_{\bar{Z}} = k\phi_{Z}k$ is the $h$-adjoint of
$\phi$. Relation (\ref{ad-j}) implies (\ref{morphism}) (easy
check) and also $\phi^{\dagger}_{\bar{Z}}=\bar{\mu}\phi_{k(Z)}$,
hence
\begin{equation}\label{cro-1}
[\phi_{Z_{1}}, \phi^{\dagger}_{\bar{Z}_{2}}]=0, \quad \forall
Z_{1}, Z_{2}\in {\mathcal T}^{1,0}_{M}.
\end{equation}
From (\ref{cro-1}) and the second $tt^{*}$-equation again, we
obtain that $D$ is flat. Consider a local frame $\{ Z_{1}, \cdots
, Z_{n}\}$ of $T^{1,0}M$, formed by $D$-parallel (holomorphic)
vector fields. Remark that $k(Z_{i})$ are also $D$-parallel and,
in particular, holomorphic. We claim that $[Z_{i}, Z_{j}]=0.$ This
follows from the first $tt^{*}$-equation, which implies
$$
0= (\partial^{D}\phi )_{Z_{i}, Z_{j}} (e) =  D_{Z_{i}}
(\phi_{Z_{j}}) (e) - D_{Z_{j}} (\phi_{Z_{i}}) (e) + [Z_{i},Z_{j}]
= [Z_{i},Z_{j}],
$$
where we used $D(Z_{i})= D(e)=0.$ Since $g(Z_{i}, Z_{j})$ is
constant (because $D(g)=0$, $D(Z_{i})=0$) and $Z_{i}$ are
holomorphic and commute, $\{ Z_{1}, \cdots , Z_{n}\}$ is the basis
of fundamental vector fields associated to a flat coordinate
system $(t^{1},\cdots , t^{n})$ for $g$. Also, $D^{(1,0)}$
coincides with the Levi-Civita connection $\nabla$ of $g$. Thus
$Z_{i}$ and $k(Z_{i})$ are $\nabla$-parallel, hence $k$ is
constant in the coordinate system $(t^{1},\cdots , t^{n})$. It
remains to show that $\circ$ is also constant in this coordinate
system, i.e. $Z_{i}\circ Z_{j}$ is $\nabla$-parallel, for any
$i,j$. For this, we apply $\bar{\partial}_{\bar{Z}}$ to relation
(\ref{morphism}). We obtain, for any $Z\in {\mathcal
T}^{1,0}_{M}$,
$$
\bar{\partial}_{\bar{Z}} \left( k (Z_{i}\circ Z_{j})\right) =
\bar{\mu}\bar{\partial}_{\bar{Z}}\left( k(Z_{i})\right)\circ Z_{2}
+ \bar{\mu}Z_{1}\circ \bar{\partial}_{\bar{Z}}\left(
k(Z_{i})\right) + \bar{\mu} \bar{\partial}_{\bar{Z}} (\circ )
\left( k(Z_{1}), k(Z_{2})\right)
$$
and the right hand side of this expression is zero
($\bar{\partial}_{\bar{Z}} (\circ ) =0$ because $\circ$ is
holomorphic). We proved that $Z_{i}\circ Z_{j}$ is $D$-parallel,
hence also $\nabla$-parallel, and the Frobenius manifold is
trivial.
\end{proof}

L. David: ``Simion Stoilow'' Institute of
the Romanian Academy, Research Unit no. 7, P.O. Box 1-764,
Bucharest, Romania; e-mail: liana.david@imar.ro

\end{document}